\documentclass[reqno]{amsart}
 \usepackage[T1]{fontenc}

\usepackage[english]{babel}
\usepackage{amsthm,amsmath,amssymb}
\usepackage{mathtools}
\usepackage{hyperref}
\usepackage{float}
\usepackage{cite}
\usepackage{textcomp}

\usepackage{graphicx}
\usepackage{subcaption}
\usepackage[export]{adjustbox}

\theoremstyle{plain}
\newtheorem{theorem}{Theorem}[section]
\newtheorem{lemma}[theorem]{Lemma}
\newtheorem{corollary}[theorem]{Corollary}
\newtheorem{proposition}[theorem]{Proposition}

\theoremstyle{definition}
\newtheorem{definition}[theorem]{Definition}

\newtheorem{remark}[theorem]{Remark}

\DeclareMathOperator{\Arc}{Arc}

\usepackage{pgf,tikz}
\usetikzlibrary{positioning,shapes,decorations.markings}
\usetikzlibrary{arrows}
\usetikzlibrary[patterns]
\usetikzlibrary{cd}

\usepackage[sc]{mathpazo} 
\usepackage{microtype} 

\usepackage{titlesec} 
\titleformat{\section}[block]{\large\scshape\centering}{\thesection.}{1em}{} 
\titleformat{\subsection}[runin]
  {\normalfont\large\bfseries}{\thesubsection}{1em}{}
\titleformat{\subsubsection}[runin]
  {\normalfont\normalsize\bfseries}{\thesubsubsection}{1em}{}

\title{Lee-Yang Zeros of the antiferromagnetic Ising Model}

\author[F. Bencs]{Ferenc Bencs$^\dag$}\address{F. Bencs: HAS Alfr\'ed R\'enyi Institute of Mathematics; Department of Mathematics, Central European University; Korteweg de Vries Institute for Mathematics, University of Amsterdam} \email{ferenc.bencs@gmail.com}
\author[P. Buys]{Pjotr Buys$^\ddagger$}
\address{P. Buys: Korteweg de Vries Institute for Mathematics, University of Amsterdam, Science Park 107, 1090GE Amsterdam, the Netherlands}
\email{pjotr.buys@gmail.com}
\author[L. Guerini]{Lorenzo Guerini$^\mathsection$}
\address{L. Guerini: Korteweg de Vries Institute for Mathematics, University of Amsterdam, Science Park 107, 1090GE Amsterdam, the Netherlands}
\email{lorenzo.guerini92@gmail.com}
\author[H. Peters]{Han Peters$^{\ddagger \, \mathsection}$}
\address{H. Peters: Korteweg de Vries Institute for Mathematics, University of Amsterdam, Science Park 107, 1090GE Amsterdam, the Netherlands}
\email{hanpeters77@gmail.com}

\thanks{$^{\dag}$
The research leading to these results has received funding from the European Research Council under the European Union’s Seventh Frame-work Programme (FP7/2007-2013) / ERC grant agreement n$^\circ$ 617747.\\
$^\ddagger$ Supported by NWO TOP grant 613.001.851.\\
$^\mathsection$ Supported by NWO TOP grant 614.001.506.}

\date{\today}

\begin{document}

\begin{abstract}
We investigate the location of zeros for the partition function of the anti-ferromagnetic Ising Model, focusing on the zeros lying on the unit circle. We give a precise characterization for the class of rooted Cayley trees, showing that the zeros are nowhere dense on the most interesting circular arcs. In contrast, we prove that when considering all graphs with a given degree bound, the zeros are dense in a circular sub-arc, implying that Cayley trees are in this sense not extremal. The proofs rely on describing the rational dynamical systems arising when considering ratios of partition functions on recursively defined trees.
\end{abstract}

\maketitle

\begin{section}{Introduction}
Partition functions play a central role in statistical physics. The distribution of zeros of the partition functions are instrumental in describing phase changes in a variety of contexts. More recently there has been a second motivation for studying the zeros of partition functions, arising from a computational complexity perspective. Since the 1990's there has been significant interest in whether the values of partition functions can be approximated, up to an arbitrarily small multiplicative error, by a polynomial time algorithm. For graphs of bounded degrees this is known to be the case  on open connected subsets of the zero free locus \cite{Ba16, PaR17}. In recent work of the last author with Regts \cite{PR17, PetersRegts2018}, the zero free locus was successfully described by first considering a specific subclass of graphs, the Cayley trees, for which the location of zeros can be described by studying iteration properties of a rational function.

A common theme in the papers \cite{PR17, PetersRegts2018} was that the Cayley trees turned out to be extremal within the larger class of bounded degree graphs, in the sense that a maximal zero free locus for Cayley trees proved to be zero-free in the larger class as well. This observation is the main motivation for our studies here, where we investigate to which extend the extremality of the class Cayley trees holds for the antiferromagnetic Ising Model.

	Let $G = (V,E)$ denote a simple graph and let $\lambda,  b  \in \mathbb{C}$. The
	\emph{partition function of the Ising model} $Z_G(\lambda,  b )$ is defined as
	\[
		Z_G(\lambda) = Z_G(\lambda,  b ) = \sum_{U \subseteq V} 	\lambda^{\lvert U \rvert} \cdot
													 b ^{\lvert \delta(U) \rvert},
	\]
	where $\delta(U)$ denotes the set of edges with one endpoint in $U$ and one endpoint
	in $U \setminus V$. In this paper we fix $ b >0$ and consider the partition function $Z_G(\lambda)$ as a polynomial in
	$\lambda$. The case $ b  < 1$ is often referred to as the \emph{ferromagnetic case},
	while $ b  > 1$ is referred to as the \emph{anti-ferromagnetic} case.

    For
	$d \geq 2$ let $\mathcal{G}_{d+1}$ be the set of all graphs of maximum degree
	at most $d+1$. Given a set of graphs $\mathcal{H}$, we write
	\[
		\mathcal{Z}_{\mathcal{H}} = \mathcal{Z}_{\mathcal{H}}( b )=
			\left\{\lambda: Z_G(\lambda) = 0 \text{ for some $G \in \mathcal{H}$} \right\}.
	\]

	When $ b  < 1$, the Lee--Yang Circle Theorem \cite{LeeYangPhase1,LeeYangPhase2} states
	that for any graph $G$, the zeros of $Z_G$ are contained in the unit circle $\mathbb{S}^1$.
	The zeros in the ferromagnetic case have subsequently been known as the Lee--Yang zeros.
	To study the zeros of $Z_G$ for all $G \in \mathcal{G}_{d+1}$ one can consider
	the subset of finite rooted Cayley trees with down degree $d$, which we
	denote by $\mathcal{C}_{d+1}$. The Lee--Yang zeros of Cayley trees are studied
	in \cite{Muller-HartmannZittartz1975,Muller-Hartmann1977,Barata1997,Barata2001,Chio2019}
	amongst other papers. In all of these papers some variation of the following
	rational function plays a important role:
	\begin{equation}
	\label{rationalmapf}
		f_\lambda(z) = f_{\lambda,d}(z) = \lambda \cdot \left( \frac{z +  b }{ b  z + 1} \right)^d,
	\end{equation}
	where $f_{\lambda}$ is viewed as a function on the Riemann sphere.
	The significance of $f_{\lambda}$ in relation to the Cayley
	trees is explained by the following lemma.
	\begin{lemma}[{e.g. \cite[Proposition 1.1]{Chio2019}}]
		\label{lem: autonomous dynamics}
		Let $ b  \in \mathbb{R}$ and $d \geq 2$, then
		\[	
			\mathcal{Z}_{\mathcal{C}_{d+1}}
			=
			\left\{\lambda: f_{\lambda}^n(\lambda) = -1
				\text{ for some $n \in \mathbb{Z}_{\geq 0}$} \right\}.
		\]
	\end{lemma}
	Thus complex dynamical systems can
	be used to study the zeros of the partition function of Cayley trees. The following result from  \cite{PetersRegts2018} shows that while the Cayley trees from a relatively small subset of the class of all graphs of bounded maximal degree, the zero free loci of these two classes are identical in the ferromagnetic case:
	\begin{theorem}
		Let $d \geq 2$. If $0 <  b  \le \frac{d-1}{d+1}$ then
		\[
			 \overline{\mathcal{Z}_{\mathcal{C}_{d+1}}}
			=
			\overline{\mathcal{Z}_{\mathcal{G}_{d+1}}} = \mathbb S^1.
		\]
		If $\frac{d-1}{d+1} <  b  < 1$ then	
		\[
		    \overline{\mathcal{Z}_{\mathcal{C}_{d+1}}}
			=
			\overline{\mathcal{Z}_{\mathcal{G}_{d+1}}}
		    =
		    \Arc[\lambda_1, \overline{\lambda_1}],
		\]
		where $\lambda_1 = \lambda_1( b ) \in \mathbb S^1$ is the unique parameter in the upper half plane for which $f_\lambda$ has a parabolic fixed point.
	\end{theorem}

Given $\alpha,\beta$ on the unit circle, we will use notation $\Arc[\alpha,\beta]$ for the closed circular arc from $\alpha$ to $\beta$, traveling counter clockwise, and similarly for open and half-open circular arcs.

When $ b  > 1$ the Lee-Yang Circle Theorem fails, and the set of zeros of the partition function is considerably more complicated. Consider for example Figure~\ref{figure:zeros}, illustrating the location of zeros for Cayley trees and for the larger class of \emph{spherically symmetric trees}, defined in Definition \ref{def1.6} below, both for maximal down-degree $d=2$ and maximal depth $11$. The pictures are symmetric with respect to reflection in the unit circle, but only few zeros outside of the unit disk are depicted because of space concerns.

\begin{figure*}[t!]
    \centering
    \begin{subfigure}[t]{0.5\textwidth}
        \centering
        \includegraphics[width=\textwidth]{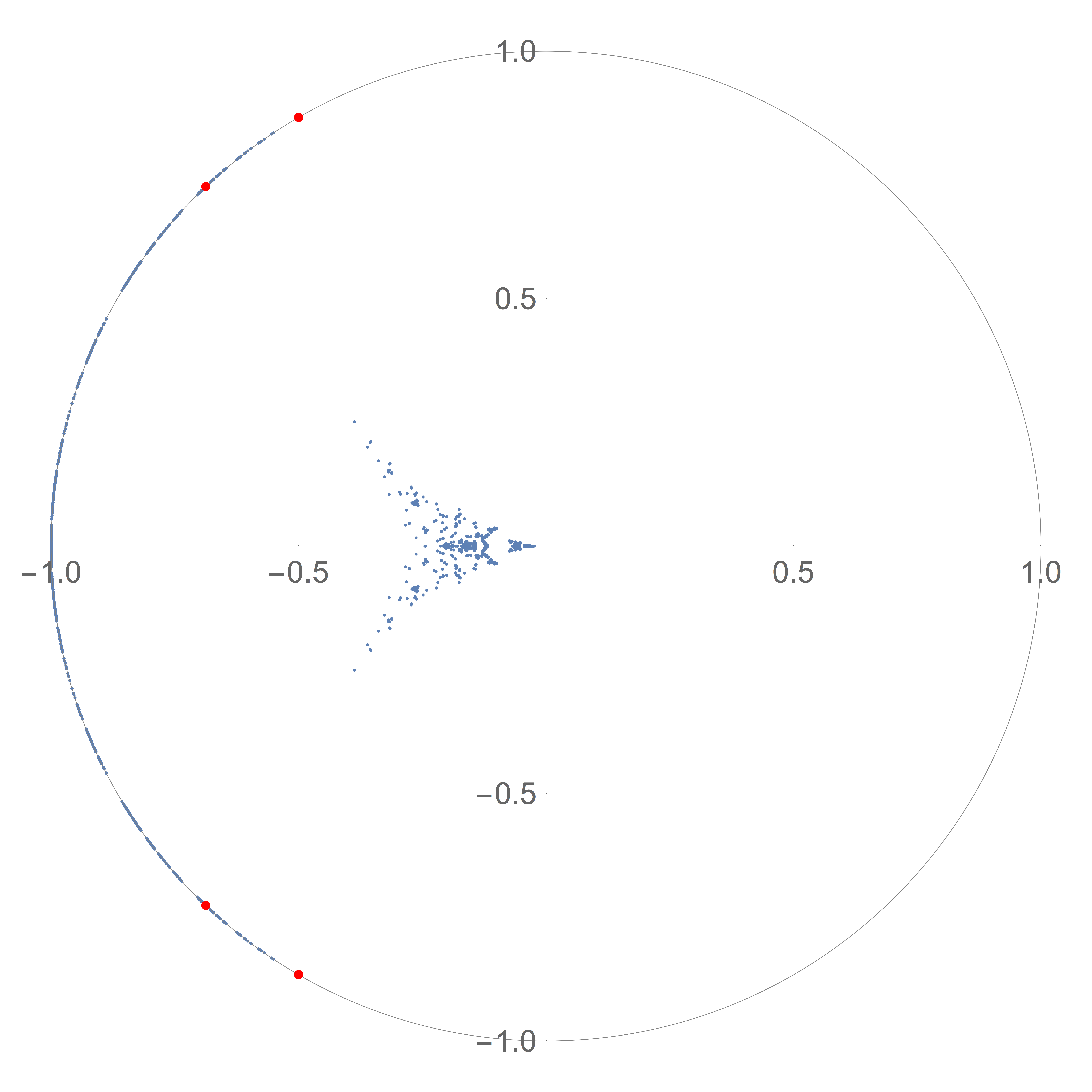}
    \end{subfigure}%
    ~
    \begin{subfigure}[t]{0.5\textwidth}
        \centering
        \includegraphics[width=\textwidth]{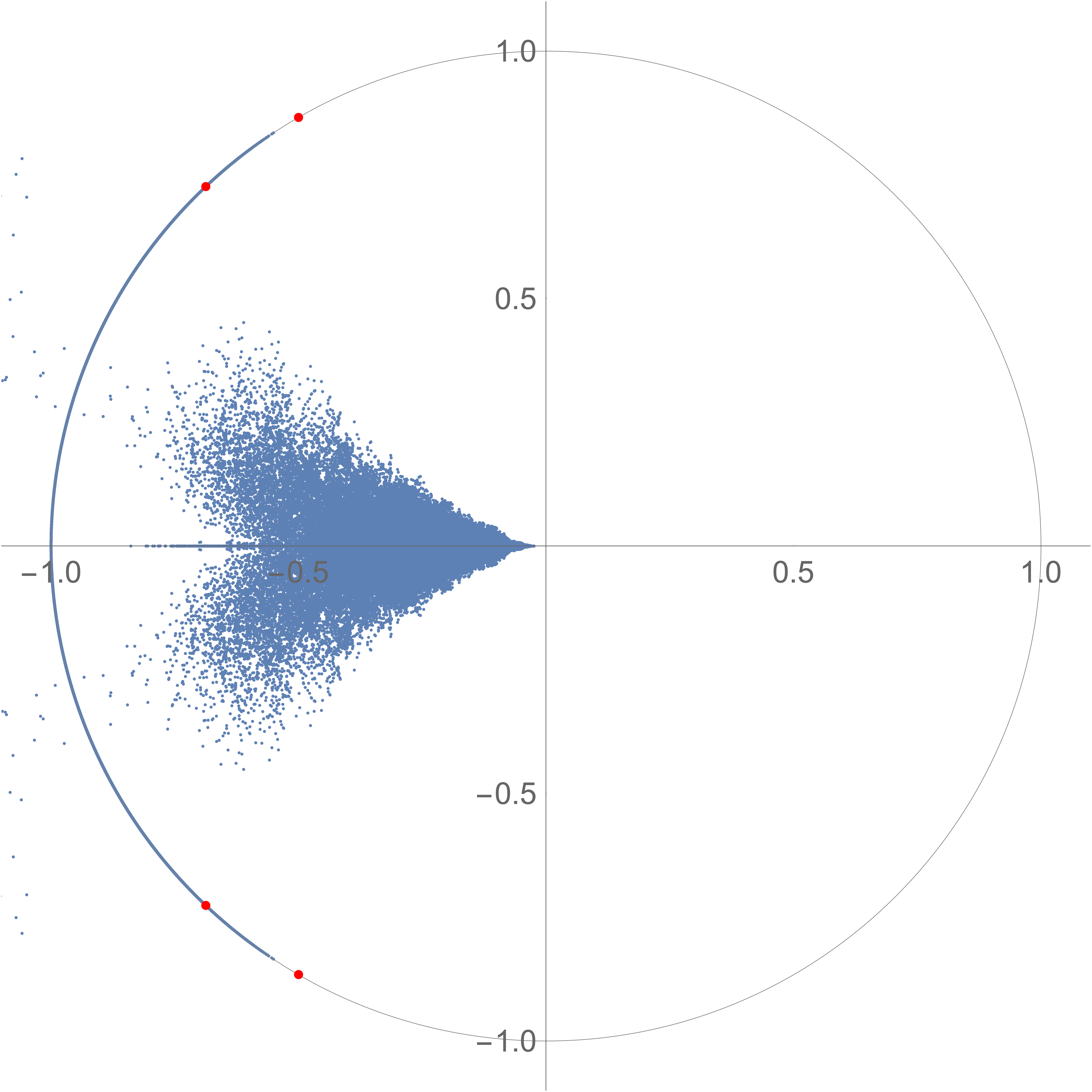}
    \end{subfigure}
    \caption{Comparing zeros of Cayley trees (left) and spherically symmetric trees (right) for $d=2$ and depth at most $11$.}
    \label{figure:zeros}
\end{figure*}

The pictures clearly demonstrate the appearance of zero parameters both on and off the unit circle. In this paper we focus on describing the set of zeros on the unit circle. Our main result show that, contrary to the ferromagnetic case, the zero free locus for the Cayley trees is strictly larger than that of the class of all bounded degree graphs.

We recall the following result from \cite{PetersRegts2018}:

\begin{theorem}
Let $\lambda_0 = e^{i \theta_0} \in \mathbb S^1$ be the parameter with the smallest positive angle $\theta_0$ for which $f_{\lambda_0}(\lambda_0) = 1$. Then
$$
\mathcal{Z}_{\mathcal{G}_{d+1}} \cap \mathbb R_+ \cdot \Arc[\overline{\lambda_0}, \lambda_0] = \emptyset,
$$
but
$$
\lambda_0, \overline{\lambda_0} \in \overline{\mathcal{Z}_{\mathcal{C}_{d+1}}}.
$$
\end{theorem}

Note that in Figure \ref{figure:zeros} $\lambda_0$ and $\overline{\lambda_0}$ are depicted by the conjugate pair of red points with smallest absolute argument. The other conjugate pair of red points corresponds to $\lambda_1$ and $\overline{\lambda_1}$, having the same definition as in the ferromagnetic case.

Our main result is the following:

	\begin{theorem} \label{thm:main}
		Let $d \ge 2$. If $ b  \ge \frac{d+1}{d-1}$ then
		\[
		\overline{\mathcal{Z}_{\mathcal{C}_{d+1}} \cap \mathbb S^1}
			=
			\overline{\mathcal{Z}_{\mathcal{G}_{d+1}} \cap \mathbb S^1} = \mathbb S^1.
		\]
		If $ 1 < b  < \frac{d+1}{d-1}$ then
\begin{enumerate}
\item {\bf  Density for Cayley trees.}
        $$
        \Arc[\lambda_1, \overline{\lambda_1}] \cup \{\lambda_0, \overline{\lambda_0} \} \subset \overline{\mathcal{Z}_{\mathcal{C}_{d+1}} \cap \mathbb S^1}.
        $$
\item {\bf Nowhere density for Cayley trees.} The set
        $$
        \overline{\mathcal{Z}_{\mathcal{C}_{d+1}}} \cap \Arc[\lambda_0, \lambda_1]
        $$
		is a nowhere dense subset of $\Arc[\lambda_0, \lambda_1]$.
\item {\bf Density for arbitrary graphs.} There exists $\lambda_3 \in \Arc(\lambda_0, \lambda_1)$ such that
		\[
			\overline{\mathcal{Z}_{\mathcal{G}_{d+1}} \cap \Arc[\lambda_0, \lambda_3]} = \Arc[\lambda_0, \lambda_3].
		\]
\end{enumerate}
	\end{theorem}
Case (3) will be proved in section 6, building upon results from earlier sections. Cases (1) and (2) will be proved respectively in sections 4 and 3.

\begin{remark}
The fact that the closure of $\mathcal{Z}_{\mathcal{C}_{d+1}}$ is strictly smaller than the closure of $\mathcal{Z}_{\mathcal{G}_{d+1}}$ also holds outside of the unit circle, a statement that is considerably easier to prove. For example, the solution to the $1$-dimensional Ising Model gives the density of zeros in a real interval $[-\alpha^{-1}, -\alpha]$, for some $\alpha\in (0,1)$. On the other hand, using Corollary \ref{corollary2comp} one can prove the existence of a neighborhood of $\lambda=-1$ where all accumulation points of $\mathcal{Z}_{\mathcal{C}_{d+1}}$ must lie on the unit circle.
\end{remark}

	We prove case (3) for the subclass in $\mathcal{G}_{d+1}$ given by the spherically symmetric trees. These trees have the advantage that dynamical methods can be used to describe the location of zeros, as indicated by the following lemma, whose proof will be given later in this section.

	\begin{lemma}\label{lemmatwo}
		Let $d \ge 2$ and $\lambda, b  \in \mathbb{C}$, then
		there exists a spherically symmetric tree $T$ with down degree at most $d$
		for which $Z_{T}(\lambda, b ) = 0$ if and only if
		\[
			g(\lambda) = -1
		\]
		for some $g \in H_{\lambda ,d+1}$.
		\end{lemma}
\end{section}
Here $H_{\lambda, d+1} = H_{\lambda,d+1}(b)$ is the rational semigroup generated by $f_{\lambda ,1}, \dots,
	f_{\lambda ,d}$, i.e.
	\[
		H_{\lambda ,d} =
			\left\langle f_{\lambda ,1}, \dots, f_{\lambda ,d}\right\rangle.
	\]
We will prove that a specific sub-semigroup of $H_{\lambda ,d}$ is hyperbolic for all $\lambda \in \Arc[\lambda_0, \lambda_2)$, for some $\lambda_2 \in \Arc[\lambda_3, \lambda_1]$, i.e. on an arc $\Arc[\lambda_0, \lambda_2]$ that contains $\Arc[\lambda_0, \lambda_2]$. Moreover, we obtain uniform bounds on the expansion rate on compact subsets of $\Arc[\lambda_0, \lambda_2)$. We also show that for any $\lambda \in \Arc[\lambda_0, \lambda_1]$, there exists a sequence in the sub-semigroup for which $+1$ lies on the Julia set. Combining these two statements we obtain uniform expansion along an orbit of $+1$. The density of zero-parameters is a consequence for $\lambda$ sufficiently close to $\lambda_0$.

We emphasize that in statement (3) of Theorem \ref{thm:main} we only consider zero parameters in $\mathbb S^1$. Alternatively we can consider $\overline{\mathcal{Z}_{\mathcal{G}_{d+1}}} \cap \mathbb S^1$, which is a priori a larger set. We prove in section 6 that this closure contains the circular arc
$$
\Arc[\lambda_0,\lambda_2),
$$
which is the arc where the earlier discussed sub-semigroup of $H_{\lambda, d}$ acts hyperbolically. The parameter $\lambda_2$ can be explicitly calculated. Computer evidence in fact suggests that
$$
\overline{\mathcal{Z}_{\mathcal{G}_{d+1}}} \cap \mathbb S^1 = \Arc[\lambda_0, \overline{\lambda_0}].
$$

In section 2 we prove basic results regarding the attracting intervals of the maps $f_\lambda$, to be used in later sections. In section 3 we consider the hyperbolic components in the parameter space of the maps $f_\lambda$, and prove case (2) of Theorem \ref{thm:main}. In section 4 we consider only parameters $\lambda$ on the unit circle and prove case (1).

In the remainder of this introduction we recall the relationship between partition functions on Cayley trees and spherically symmetric trees on the one hand, and respectively iteration and semi-group actions on the other hand. In particular we give a short proof of Lemmas \ref{lem: autonomous dynamics} and \ref{lemmatwo}. In the rest of the paper we will only consider the two dynamical systems, with few references to partition functions.

\subsection{Iterates and semigroups arising from trees}

Let us recall from \cite{PetersRegts2018} (but see also \cite{Chio2019}) how the zeros of the Ising partition function $Z_G(\lambda)$ on some recursively defined trees can be studied using iterations or compositions of rational functions.

Let $v$ be a marked node of a graph $G = (V,E)$. Note that
$$
Z_G = Z_{G,v}^{in} + Z_{G,v}^{out},
$$
where $Z_{G,v}^{in}$ sums only over $U \subset V$ with $v \in U$, and $Z_{G,v}^{out}$ sums only over $U \subset V$ with $v \notin U$. It follows that
$$
Z_G = 0 \; \Leftrightarrow \; R_{G,v} := \frac{Z_{G,v}^{in}}{Z_{G,v}^{out}} = -1 \; \mathrm{or} \; Z_{G,v}^{in} = Z_{G,v}^{out} = 0.
$$

Suppose now that $G=T$ is a tree. Denote the neighbors of $v$ by $v_1, \ldots, v_k$, and the corresponding connected components of $G-v$ by $T_1,  \ldots , T_k$. Then it follows that
$$
R_{T,v} = \lambda \prod_{i=1}^k \frac{R_{T_i, v_i} +  b }{ b  R_{T_i, v_i} + 1}.
$$
Hence when all the rooted trees $(T_i, v_i)$ are isomorphic, one obtains
$$
R_{T,v} = \lambda \left( \frac{R_{T_i, v_i} +  b }{ b  R_{T_i, v_i} + 1} \right)^k.
$$

\begin{definition}\label{def1.6}
Let $d \ge 2$ and let $\omega = (k_1, k_2, \ldots) \in \{1, \ldots , d\}^\mathbb N$. Let $T_0$ be the rooted graph with a single vertex. Recursively define the trees $T_1, \ldots$ by letting $T_n$ consist of a root vertex $v$ of degree $k_n$, with each edge incident to $v$ connected to the root of a copy of $T_{n-1}$. We say that the rooted trees $T_n$ are \emph{spherically symmetric} of degree at most $d$. Equivalently a rooted tree, with root $v$, is said to be spherically symmetric if all leaves have the same depth $n$, and all vertices of depth $1\le j<n$ have down-degree $k_j$. When all degrees $k_n$ are equal to $d$ the tree $T_{n}$ is said to be a \emph{(rooted) Cayley tree} of degree $d$.
\end{definition}

Note that for a spherically symmetric tree
\begin{align*}
  Z^{in}_n(\lambda)&=\lambda (Z^{in}_{n-1}(\lambda)+ b  Z^{out}_{n-1}(\lambda))^{k_n} \in \mathbb{R}[\lambda], \; \mathrm{and}\\
  Z^{out}_n(\lambda)&=( b  Z^{in}_{n-1}(\lambda)+Z^{out}_{n-1}(\lambda))^{k_n} \in \mathbb{R}[\lambda].
\end{align*}
Since we will work with $ b  \notin \{-1, +1\}$ it follows by induction that $Z^{in}_n(\lambda)$ and $Z^{in}_n(\lambda)$ cannot both be equal to zero, from which it follows that
$$
Z_G = 0 \; \Leftrightarrow \; R_{G,v} = -1.
$$

Noting that $R_{T_0,v} = \lambda$, it follows for Cayley trees that
$$
R_{T_n,v} = f_d^n(\lambda) = f_d^{n+1}(+1),
$$
where
$$
f(z) = f_d(z)=\lambda\left(\frac{z+ b }{ b  z+1}\right)^d,
$$
while for spherically symmetric trees we obtain
$$
R_{T_n,v} = f_\omega^n(\lambda) := f_{k_n} \circ \cdots \circ f_{k_1} (\lambda).
$$
Hence we have proved lemmas \ref{lem: autonomous dynamics} and \ref{lemmatwo}.

Motivated by this discussion we introduce the notations
$$
\mathcal Z_f := \{ \lambda \in \mathbb C \,:\, f^n(1) = -1 \; \text{for some} \; n \in \mathbb N \}.
$$
and
$$
\mathcal Z_{H} := \{\lambda \in \mathbb C \,:\, f^n_\omega(1) = -1 \; \text{for some} \; n \in \mathbb N, \omega \in \{1, \ldots, d\}^{\mathbb N} \},
$$
where $H$ again refers to the semi-group $\langle f_1, \ldots, f_d\rangle$. Thus $\lambda \in \mathcal{Z}_f$ if and only if $Z_G(\lambda) = 0$ for a Cayley tree $G$, while $\lambda \in \mathcal{Z}_{H}$ if and only if $Z_G(\lambda) = 0$ for a spherically symmetric tree $G$.

\begin{section}[J-stable components]{$J$-stable components}


Given a family of rational maps $f_\lambda$ parameterized by a complex manifold $\Lambda$,
the set of $J$-stable parameters is the set of parameters for which the Julia set $J_\lambda$
moves continuously with respect to the Hausdorff topology. The concept of $J$-stability plays
a central role in the study of rational functions. We refer the interested reader to
\cite{MSS,Mc,Sl} for a more detailed description of $J$-stability.

Given a positive integer $d\ge 2$ and $b>1$, let $f_\lambda$ be the family of rational functions
given by \eqref{rationalmapf} parameterized by $\Lambda = \widehat{\mathbb C}$. We will write
$\Lambda^{stb}$ for the set of $J$-stable parameters and $\Lambda^{hyp}$ for the set of hyperbolic
parameters, i.e. the values for which $f_\lambda$ has no critical points nor parabolic cycles on $J_\lambda$.
Recall that $\Lambda^{stb}$ is a dense open set and that the set $\Lambda^{hyp}$ is an open and closed subset of
$\Lambda^{stb}$. Whether the equality $\Lambda^{stb} = \Lambda^{hyp}$ holds for the family given by (1) is a natural question, though not directly relevant for our purposes.

Given $\lambda \in \Lambda^{stb}$ we will write $\Lambda^{stb}_\lambda$ for the connected component
of $\Lambda^{stb}$ containing the parameter $\lambda$.

\begin{theorem}
	\label{holomorphicmotionthm}
	There exists a holomorphic motion of $J_{\lambda}$ over $(\Lambda_{\lambda}^{stb},\lambda)$ which
	respects the dynamics, i.e. there exists a continuous map $\varphi: \widehat{\mathbb C} \times
	\Lambda^{stb}_{\lambda} \rightarrow \hat{\mathbb C}$ satisfying
	\begin{itemize}
		\item[1.]
		$\varphi_\mu$ is the identity at the base point $\lambda$, i.e. $\varphi_{\lambda}(z)=z$,
		\item[2.]
		for every $z\in J_{\lambda}$ the map $\varphi_\mu(z)$ is holomorphic in $\mu \in \Lambda^{stb}_{\lambda}$,
		\item[3.]
		for every $\mu \in \Lambda^{stb}_{\lambda}$ the map $z \mapsto \varphi_\mu(z)$ is injective and
		can be extended to a quasi-conformal map $\varphi_\mu:\hat{\mathbb C} \rightarrow \hat{\mathbb C}$.
		\item[4.]
		for every $\mu\in \Lambda^{stb}_{\lambda}$ the map $\varphi_\mu: J_{\lambda}\rightarrow J_{\mu}$ is
		a homeomorphism. Furthermore, the following diagram commutes
		$$
			\begin{tikzcd}[sep=large]
				J_{\lambda}\arrow[r, "f_{\lambda}"] \arrow[d,"\varphi_\mu",swap]
				& J_{\lambda}\arrow[d,"\varphi_\mu"] \\
				J_\mu \arrow[r, "f_\mu"]
				& J_{\mu}
			\end{tikzcd}
		$$
		\end{itemize}
	Such map $\varphi$ satisfies the following two additional properties
		\begin{itemize}
		\item[5.]
		Given $z \in \widehat{\mathbb C}$ the map $\varphi^{-1}_\mu(z)$ is continuous with respect to $\mu$,
		\item[6.]
		Let $z_n\to z$ be a convergent sequence and assume that $\mu\mapsto\varphi_\mu(z)$ is not constant.
		Then there exists a subsequence $n_k$ and $\mu_k\to\lambda$ so that
		$$
			\varphi_{\mu_k}(z_{n_k})=z.
		$$
	\end{itemize}
\end{theorem}
\begin{remark}
The existence of a continuous map $\varphi$ satisfying properties $1-4$ was proven by \cite{MSS,Sl},
while the properties $5$ and $6$ follow immediately from continuity of $\varphi$. The holomorphic
motion is unique on the Julia set $J_\lambda$, in the sense that any other continuous map $\widetilde \varphi$
which satisfies the properties $1-4$ has to agree with $\varphi$ on the set $J_\lambda \times \Lambda^{stb}_\lambda$ .
\end{remark}

Define the two sets
$$
	\mathcal F \coloneqq \{\lambda \in \widehat{\mathbb C}\,|\, 1\in F_\lambda\},
		\qquad
	\mathcal J \coloneqq \{\lambda \in \widehat{\mathbb C}\,|\, 1\in J_\lambda\}.
$$
Given a connected component $U \subset \Lambda^{stb}$ we further write
$\mathcal F_U = \mathcal F \cap U$ and $\mathcal J_U = \mathcal J \cap U$. Since the
Julia set $J_\lambda$ moves continuously for $\lambda \in U$ it follows that $\mathcal F_U$
is open while $\mathcal J_U$ is closed with respect the intrinsic topology of $U$.

\begin{definition}
Let $U$ be a connected component of $\Lambda^{stb}$. We say that $U$ is \emph{exceptional} if
there exists $\lambda \in \mathcal J_U$ so that $\mu \mapsto \varphi_\mu(1)$ is constant,
where $\varphi_\mu$ denotes the holomorphic motion of $J_\lambda$ over $(U,\lambda)$.
\end{definition}

\begin{remark}\label{exceptional}
Suppose that $U$ is an exceptional component of $\Lambda^{stb}$ and let $\lambda \in U$ be so that the
map $\mu \mapsto \varphi_\mu(1)$ is constant. Given another $\widetilde\lambda \in U$ we have that
$\varphi_{\widetilde\lambda}(1) = \varphi_{\lambda}(1) =1$, and therefore that $1 \in J_{\widetilde \lambda}$.
Let $\widetilde\varphi_\mu$ be the holomorphic motion of $J_{\widetilde \lambda}$ over $(U,\widetilde \lambda)$.
Then we have
$$
	\widetilde \varphi_\mu(1) = \varphi_\mu\circ\varphi_{\widetilde\lambda}^{-1}(1) = 1,
	\qquad
	\forall \mu\in U.
$$
This shows that  if $U$ is an exceptional component then $\mathcal J_U = U$ and for every
$\lambda \in U$ the map $\mu \mapsto \varphi_\mu(1)$ is constant.
\end{remark}

\begin{proposition}\label{perfect}
Let $U$ be a connected component of $\Lambda^{stb}$. Then the set $\mathcal J_U$ is perfect
with respect to the intrinsic topology of $U$.
\end{proposition}
\begin{proof}
We already know that $\mathcal J_U$ is closed in $U$, thus we only have to show that $\mathcal J_U$
contains no isolated points. If $U$ is an exceptional hyperbolic component, then according to
Remark~\ref{exceptional}  we have $\mathcal J_U = U$ and the result follows immediately. Assume
instead that $U$ is not exceptional, let $\lambda \in \mathcal J_U$ and $\varphi_\mu$ be the holomorphic
motion of $J_{\lambda}$ over $(U,\lambda)$.

Since the Julia set of a rational map is perfect, it follows that we may take $z_n \in J_{\lambda}$ which
converges to $1$, and that is not identically equal to $1$. By Theorem~\ref{holomorphicmotionthm} we
may therefore find a sequence $n_k \ge 0$ and $\mu_k \to \lambda$ so that $\varphi_{\mu_k}(z_{n_k}) = 1$ for
every $k$. Since $\varphi_\mu(J_{\lambda}) = J_\mu$ we conclude that $\mu_k \in \mathcal J_U$, proving
that $\lambda$ is not an isolated point of $\mathcal J_U$, and that $\mathcal J_U$ is perfect.
\end{proof}

The definition of active parameters is classical \cite{Mc2,DF}, and was inspired by  \cite{Lev,Ly}. In all these
works activity is always defined in terms of the family $\{f^n_\lambda\circ c(\lambda)\}$, where $c(\lambda)$
is the parameterization of a critical point. For our purpose it is natural to replace $c(\lambda)$ with
the point $1$, even though the point $1$ is never critical.
\begin{definition}
A parameter $\lambda \in \widehat{\mathbb C}$ is \emph{passive} if the family
$\{\lambda\mapsto f^n_\lambda(1)\}_{n\in\mathbb N}$ is normal in some neighborhood of $\lambda$, and
is \emph{active} otherwise.
\end{definition}
We further remark that, given a marked point $a(\lambda)$ and the corresponding family
$\{f^n_\lambda \circ a(\lambda)\}$, it would be more accurate to say that \emph{the marked point $a(\lambda)$ is
passive/active at $\lambda$}. However since in our case the marked point is always $1$, we will refer to
passive/active parameters instead.

\begin{lemma}
\label{closureZ}
Every active parameter is in $\overline {\mathcal Z_{\mathcal C_{d+1}}}$.
\end{lemma}
\begin{proof}
This is a standard normality argument. Assume first that $d \neq 2$ or that $\lambda\neq -1$. Let
$\lambda$ be an active parameter and choose $\alpha_0, \beta_0 \in f_\lambda^{-1}(\{-1\})$ so that
$\{-1,\alpha_0,\beta_0\}$ are all distinct. Since $-1$ is never a critical value of $f_\lambda$ we can define two
holomorphic map $\alpha_\mu,\beta_\mu$ so that $f_\mu(\{\alpha_\mu,\beta_\mu\})=-1$ in a neighborhood
of $\lambda$. By conjugating with a holomorphically varying family of M\"obius transformation we may assume
that $\{-1,\alpha_\mu,\beta_\mu\}=\{-1,0,\infty\}$. Since the family $\{f^n_\mu(1)\}$ is not normal at $\lambda$,
by Montel's Theorem we conclude that it cannot avoid the three points $\{-1,0,\infty\}$ in a neighborhood of
$\lambda$. Since $\{0,\infty\}$ are both mapped to $-1$, the orbit $f^n_\lambda(1)$ cannot miss the point $-1$
near $\lambda$, proving that $\lambda\in\overline{\mathcal Z_{\mathcal C_{d+1}}}$. When $d=2$ and
$\lambda = -1$ the point $-1$ is fixed and has only two preimages $\{1,-1\}$. In this case we fix $\alpha_0=1$
and we choose $\gamma_0$ as one of the two preimages of $1$, the proof is then the same as above.
\end{proof}

\begin{lemma}\label{normalfamily}
Let $U$ be a non-exceptional component of $\Lambda^{stb}$. Then every $\lambda\in \mathcal F_U$ is
passive and every $\lambda \in \mathcal J_U$ is active.
\end{lemma}
\begin{proof}
Given $\lambda \in U$ let $\varphi_\mu$ be the holomorphic motion of $J_\lambda$ over $(U,\lambda)$.
If $\lambda \in \mathcal F_U$ then the orbit $f_\lambda^n(1)$, avoids the Julia set $J_\lambda$ and
in particular it avoids three distinct points $\{a,b,c\}\subset J_\lambda$. Since the set $\mathcal F_U$ is open we
have that $1 \in F_\mu$ for every $\mu$ sufficiently close to $\lambda$, and therefore the orbit of $f^n_\mu(1)$
avoids the set $\varphi_\mu(\{a,b,c\}) \subset J_\mu$. Using the normality argument from the proof of the previous lemma, we may therefore conclude that $\{f_\mu^n(1)\}$ is normal in a neighborhood of $\lambda$,
showing that $\lambda$ is passive.

Suppose that there exists $\lambda \in \mathcal J_U$ which is passive. Given any
$0 < \varepsilon < \textrm{diam}(J_\lambda)/2$ by equicontinuity we can find $\delta>0$ so that
$$
	\left| f^n_\mu(1) - f^n_\lambda(1)\right| < \varepsilon/2,
	\qquad
	\forall \mu \in B(\lambda,\delta) \text{ and } n\in\mathbb N.
$$

Given any open neighborhood $U \ni 1$ there exists $N \in \mathbb N$ so that
$f^N_\lambda(U\cap J_\lambda) = J_\lambda$. Given $w \in J_\lambda$ with distance at least
$\varepsilon$ from $f^N_\lambda(1)$ we can therefore find $z \in U \cap J_\lambda$ so that $f^N_\lambda(z)=w$.
We conclude that we can construct a sequence $z_k \in J_\lambda$ converging to the point $1$ and a
sequence of positive integers $N_k$ so that
\begin{equation}
	\label{expansioncontradiction}
	\left|f^{N_k}_\lambda(z_k) - f^{N_k}_\lambda(1)\right| > \varepsilon.
\end{equation}

By Theorem \ref{holomorphicmotionthm}, up to taking a subsequence of $z_k$ if necessary, we may
assume that there exists a sequence $\mu_k \in B(\lambda,\delta)$ so that $\mu_k \to \lambda$ and so
that $\varphi_{\mu_k}(z_k) = 1$. This implies that
$$
	\left|\varphi_{\mu_k} \circ f^{N_k}_\lambda(z_k) - f^{N_k}_\lambda(1)\right|
		=
	\left|f^{N_k}_{\mu_k}(1) - f^{N_k}_\lambda(1)\right|
		<
	\varepsilon/2.
$$
By continuity of the holomorphic motion, we may further assume that whenever $k$ is sufficiently large we
have
$$
	\left|\varphi_{\mu_k} \circ f^{N_k}_\lambda(z_k) - f^{N_k}_\lambda(z_k)\right|
		=
	\left|\varphi_{\mu_k}\circ f^{N_k}_\lambda(z_k) - \varphi_{\lambda} \circ f^{N_k}_\lambda(z_k)\right|
		<
	\varepsilon/2.
$$
In combination with the previous inequality we conclude that for every $k$ sufficiently large we have
$$
	\left|f^{N_k}_{\lambda}(z_k) - f^{N_k}_\lambda(1)\right|
		<
	\varepsilon,\qquad\forall n\in\mathbb N,
$$
contradicting the definition of the sequence $z_k$. Thus every $\lambda\in\mathcal J_U$ is active.
\end{proof}

\end{section}


\begin{section}[Dynamics of the map f]{Dynamics of the map $f_\lambda$}
\label{sec:dynamicsf}

For given $d \ge 2$ and $ b > 1$, we are interested in the dynamics of the map $f_\lambda$ under the
assumption that $\lambda \in \mathbb S^1$. In this case we have
\begin{align*}
	&f_\lambda:B(0,1)\rightarrow\widehat{\mathbb C}\setminus\overline{B(0,1)}\rightarrow B(0,1),\\
	&f_\lambda:\mathbb S^1\rightarrow\mathbb S^1,
\end{align*}
and the restriction of $f$ to $\mathbb S^1$ is orientation reversing. If we write $F_\lambda$ and $J_\lambda$ for the
Fatou and the Julia set of the map $f_\lambda$ we conclude that
\begin{equation}
	\label{basicjuliasetinclusion}
	F_\lambda \supset \hat{\mathbb C} \setminus \mathbb S^1,
		\qquad
	J_\lambda\subset \mathbb S^1,
		\qquad
	\forall \lambda \in \mathbb S^1.
\end{equation}
When $\lambda\in\mathbb S^1$, we further have
\begin{equation}
	\label{derivative1}
	\left|f'_\lambda(z)\right|
		=
	\frac{d(b ^2-1)}{1+ b ^2+2 b \,\mathrm{Re}\, z},
		\qquad
	\forall z\in\mathbb S^1.
\end{equation}
Therefore the value of $|f'_\lambda(z)|$ increases as $\textrm{Re}\,z$ decreases. Recall that a rational
map $f$ is \emph{expanding} on an invariant set $K$ if $f$ locally increases distances, while it is
\emph{uniformly expanding} if distances are locally increased by a multiplicative factor, bounded below
by a constant strictly greater than $1$.

\begin{lemma}[{\cite[Lemma 9]{PetersRegts2018}}]
\label{atb=bordello}
If $b > \frac{d+1}{d-1}$ and $\lambda \in \mathbb S^1$ then the map $f_\lambda$ is uniformly expanding
on $\mathbb S^1$. If $b = \frac{d+1}{d-1}$ then the map $f_\lambda$ is expanding on $\mathbb S^1$.
\end{lemma}

\begin{definition}[{\cite[Lemma 12]{PetersRegts2018}}]
Given $1 < b<\frac{d+1}{d-1}$ we write $\lambda_1 \in \mathbb S^1$ for the unique parameter satisfying
 $0< \mathrm{arg}(\lambda_1) < \pi$ and for which $f_{\lambda_1}$ has a parabolic fixed point.
\end{definition}

The following proposition describes the set of hyperbolic parameters on the unit circle
\begin{proposition}
\label{noneutral}
We have
	$$
	\mathbb S^1\cap\Lambda^{hyp}
		=
	\begin{cases}
		\mathbb S^1 						&\text{if }b>\frac{d+1}{d-1},\\
		\mathbb S^1\setminus\{1\}				&\text{if }b=\frac{d+1}{d-1},\\
		\mathbb S^1\setminus\{\lambda_1,\overline\lambda_1\}	&\text{if }1<b<\frac{d+1}{d-1}.
	\end{cases}
$$
\end{proposition}
\begin{proof}
When $b>\frac{d+1}{d-1}$ then by Lemma \ref{atb=bordello} the map $f_\lambda$ is uniformly
expanding and therefore hyperbolic. When $b = \frac{d+1}{d-1}$ and $\lambda \in \mathbb S^1\setminus \{1\}$
then for every $z \in \mathbb S^1$ either $z$ or $f(z)$ is uniformly bounded away from $1$. By \eqref{derivative1}
we obtain again that the map $f^2_\lambda$ is uniformly expanding, proving that $f_\lambda$ is hyperbolic.
On the other hand when $\lambda = 1$ the map $f_\lambda$ has a parabolic fixed point, and therefore it is not
       hyperbolic.

Given $1<b<\frac{d+1}{d-1}$, then $\lambda_1,\overline\lambda_1$ are the unique parameters on the unit circle
for which $f_\lambda$ has a parabolic fixed point. Suppose that there exists
$\lambda \in \mathbb S^1 \setminus \{\lambda_1,\overline\lambda_1\}$ which is not hyperbolic. By
\eqref{basicjuliasetinclusion} the set $\widehat{\mathbb C} \setminus \mathbb S^1$ is contained in the Fatou set, and
therefore the critical points of $f_\lambda$ are also contained in the Fatou set. It follows that the map $f_\lambda$
must have a parabolic cycle with period at least $2$. Since there are at most two Fatou components we conclude
that the period of the parabolic cycle is exactly $2$.

Notice that for every $\lambda \in \mathbb S^1$ we have that $f_\lambda\left(1/\overline z\right) = 1/\overline{f_\lambda(z)},$
and therefore that
$$
	f^n_\lambda(-1/ b ) = 1/\overline{f^n_\lambda(- b )}
$$
Let $z_1,z_2\in\mathbb S^1$ be the parabolic cycle of $f_\lambda$. These two points are parabolic fixed points
for $f^2_\lambda$ and both of them have an immediate basin that must coincide with a Fatou component of
$f_\lambda$. By replacing $z_1$ with $z_2$ if necessary, we may therefore assume that $B(0,1)$ is the attracting
basin of $z_1$, while $\widehat{\mathbb C}\setminus\overline{B}(0,1)$ is the attracting fixed point of $z_2$. This
shows that $f^{2n}_\lambda(-1/ b )\to z_1$ and that $f^{2n}_\lambda(- b )\to z_2$. And therefore that
$z_1=1/\overline z_2=z_2$, contradicting the fact that the period of the cycle is $2$.
\end{proof}

We notice that $-1$ is always a hyperbolic parameter, therefore the set $\Lambda^{hyp}_{-1}$, i.e., the
connected component of $\Lambda^{hyp}$ containing $-1$,  is always well defined. On the other hand the
set $\Lambda^{hyp}_1$ is defined for $b\neq \frac{d+1}{d-1}$ and does not coincide with $\Lambda^{hyp}_{-1}$
if and only if $1< b < \frac{d+1}{d-1}$.

\begin{proposition}
\label{generaljulia}
Let $b>1$. Then for every $\lambda \in \Lambda^{hyp}_{-1}$ the Julia set $J_\lambda$ is a quasi-circle, while the
Fatou set $F_\lambda$ contains exactly two components which are the attracting basin of a (super)attracting $2$-cycle.
If we further assume that $|\lambda| = 1$ then $J_\lambda = \mathbb S^1$.

Let $1 < b < \frac{d+1}{d-1}$. Then for every $\lambda \in \Lambda^{hyp}_1$ the Julia set $J_\lambda$ is
a Cantor set, while the Fatou set $F_\lambda$ coincides with the attracting component of an attracting fixed point.
\end{proposition}
\begin{proof}
The function $g(z) = f^2_{-1}(z) - z$ satisfies $g(0)<0$ and $g(-1/b)>0$. Since $g$ is a real map and $f_{-1}$ maps
the disk to the complement of its closure, we conclude that $f_{-1}$ has a periodic point of order 2 in $B(0,1)$. By
\eqref{basicjuliasetinclusion} it is clear that $J_{-1}=\mathbb S^1$. The holomorphic motion of $J_{-1}$ over
$(\Lambda^{hyp}_{- 1},-1)$ given by Theorem \ref{holomorphicmotionthm} now implies that the Julia set $J_\lambda$
is a quasi-circle for every $\lambda \in \Lambda^{hyp}_{-1}$. The two components of $F_{-1}$ are mapped into
each other, and by continuity the same holds for $F_\lambda$. Hyperbolicity of $f_\lambda$ implies that they are
the basin of a (super)attracting $2$-cycle.

When $1 < b < \frac{d+1}{d-1}$ the map $f_1$ has an attracting fixed point at $1$. It is well known that the Julia
set of a rational function with a single invariant attracting basin containing all the critical points is a Cantor set
(see also \cite[Theorem B.1]{Mi}). Proceeding as above we obtain that for every $\lambda \in \Lambda^{hyp}_1$
the set $J_\lambda$ is a Cantor set and that $F_\lambda$ coincides with the attracting basin of a (super)attracting
fixed point. Since the critical point of $f_\lambda$ are not fixed point, we conclude that the fixed point is attracting.
\end{proof}

\begin{remark}
A \emph{bicritical} rational map is a rational with two distinct critical points (counted without multiplicity).
The space of bicritical rational map of degree $d$ was studied by Milnor \cite{Mi}, where he shows that its Moduli
space (the space of holomorphic conjugacy classes) is biholomorphic to $\mathbb C^2$. In this paper he constructs
explicit conjugacy invariants $f\mapsto (X,Y)$. In our case the invariants associated to the map $f_\lambda$ are
given by
$$
	X = \frac{b^2}{1-b^2},
		\qquad
	Y = \left(\lambda+\frac{1}{\lambda}\right)\frac{b^{d-1}}{(1-b^2)^d}.
$$
A bicritical rational map is \emph{real} if its invariants are real, or equivalently if there exists an antiholomorphic
involution $\alpha$ which commutes with the map. When $b \in \mathbb R \setminus\{0,1\}$, the map $f_\lambda$ is real
if and only if $\lambda \in \mathbb S^1$, and the corresponding involution is $\alpha = 1/\overline z$. The results obtained
by Milnor for real maps are sufficient to conclude that given $\lambda \in \mathbb S^1$ the Julia set is either a Cantor set
or the whole circle.
\end{remark}




The following definition follows from the proposition above. Recall that when $1 < b < \frac{d+1}{d-1}$ by
Proposition~\ref{noneutral} we have $\Arc(\overline{\lambda_1},\lambda_1) = \Lambda^{hyp}_1 \cap \mathbb S^1$ and that when
$\lambda \in \mathbb S^1$ all fixed point of $f_\lambda$ are on the unit circle.
\begin{definition}
Given $1 < b < \frac{d+1}{d-1}$ and $\lambda \in \Arc(\overline{\lambda_1},\lambda_1)$, we write $R_\lambda\in\mathbb S^1$
for the attracting fixed point of $f_\lambda$ and $I_\lambda$ for the connected component of $F_\lambda \cap\mathbb S^1$
containing $R_\lambda$. Notice that the map $f_\lambda$ is an orientation reversing bijection
$f_\lambda: I_\lambda \rightarrow I_\lambda$.
\end{definition}


\begin{theorem}
\label{l0l1theorem}
Let $1 < b < \frac{d+1}{d-1}$. Then there exist unique parameters $\lambda_0 \in \mathbb S^1$ with
$0 < \arg(\lambda_0) < \arg(\lambda_1) < \pi$, so that when $\lambda \in \Arc[1,\lambda_1)$
$$
	\begin{cases}
	\Arc(1,\lambda) 	\Subset I_\lambda,			&\textit{for }\lambda\in \Arc[1,\lambda_0),\\
	I_\lambda 			= \Arc(1,\lambda),			&\textit{for }\lambda=\lambda_0,\\
	I_\lambda 			\Subset \Arc(1,\lambda),	&\textit{for }\lambda\in \Arc(\lambda_0,\lambda_1).
	\end{cases}
$$
Similar inclusions hold for $\lambda\in(\overline\lambda_1,1]$.
\end{theorem}

\begin{figure}[ht]
	\centering
	\begin{tikzpicture}[line cap=round,line join=round,>=triangle 45,x=2.5cm,y=2.5cm]
		\clip(-1.2,-1.2) rectangle (1.2,1.2);
		\draw (0,0) circle (1);
		\draw [|-,shift={(0,0)},dash pattern=on 1pt off 1pt]
					plot[domain=-1.59:1.50,variable=\t]({1*0.9*cos(\t r)+0*0.9*sin(\t r)},{0*0.9*cos(\t r)+1*0.9*sin(\t r)});
		\draw [-|,shift={(0,0)},dash pattern=on 1pt off 1pt]
					plot[domain=1.64:2.31,variable=\t]({1*0.9*cos(\t r)+0*0.9*sin(\t r)},{0*0.9*cos(\t r)+1*0.9*sin(\t r)});
		\begin{scriptsize}
			\fill  (0.54,0.84) circle (1.5pt);
			\draw[above right] (0.54,0.84) node {$\lambda$};
			\fill  (-0.68,0.74) circle (1.5pt);
			\fill  (-0.02,-1) circle (1.5pt);
			\fill  (1,0) circle (1.5pt);
			\draw[right] (1,0) node {$1$};
			\draw (0,.89) node {$I_\lambda$};
			\fill (0.828425,0.558646) circle (1.5pt);
			\draw[above right] (0.828425,0.558646) node {$R_\lambda$};
		\end{scriptsize}
	\end{tikzpicture}
	\begin{tikzpicture}[line cap=round,line join=round,>=triangle 45,x=2.5cm,y=2.5cm]
		\clip(-1.2,-1.2) rectangle (1.2,1.2);
		\draw (0,0) circle (1);
		\draw [|-,shift={(0,0)},dash pattern=on 1pt off 1pt]
					plot[domain=0.41:1.5,variable=\t]({1*0.9*cos(\t r)+0*0.9*sin(\t r)},{0*0.9*cos(\t r)+1*0.9*sin(\t r)});
		\draw [-|,shift={(0,0)},dash pattern=on 1pt off 1pt]
					plot[domain=01.64:1.94,variable=\t]({1*0.9*cos(\t r)+0*0.9*sin(\t r)},{0*0.9*cos(\t r)+1*0.9*sin(\t r)});
		\begin{scriptsize}
			\fill (-0.6,0.8) circle (1.5pt);
			\draw[above left] (-0.6,0.8) node {$\lambda$};
			\fill (-0.36,0.93) circle (1.5pt);
			\fill (0.92,0.4) circle (1.5pt);
			\fill (1,0) circle (1.5pt);
			\draw[right] (1,0) node {$1$};
			\draw (0,.89) node {$I_\lambda$};
			\fill (-0.6,0.8) circle (1.5pt);
		\end{scriptsize}
	\end{tikzpicture}
	\caption{The position of $I_\lambda$ for $\lambda\in \Arc(1,\lambda_0)$ and for $\lambda \in \Arc(\lambda_0,\lambda_1)$.}
\end{figure}
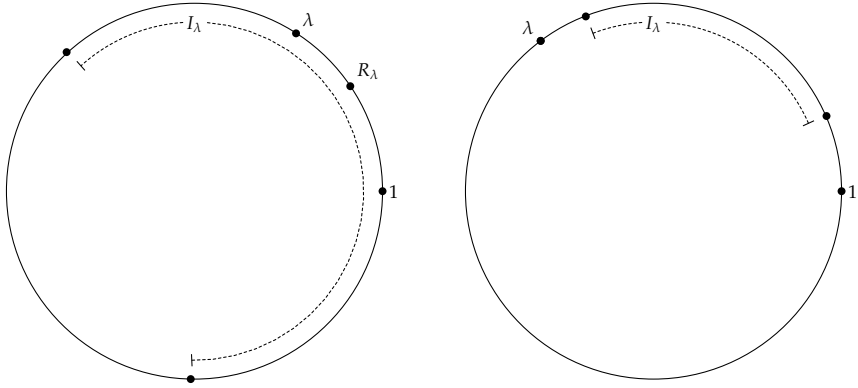

The existence of $\lambda_0$ and the first two inclusions follows from \cite[Theorem 5]{PetersRegts2018}. Since the dynamics
of $f_\lambda$ is conjugate to the dynamics of $f_{\overline{\lambda}}$ it will be sufficient to prove that
$I_\lambda \Subset \Arc(1,\lambda)$ for $\lambda \in \Arc(\lambda_0,\lambda_1)$.

By the implicit function theorem the point $R_\lambda$ moves holomorphically in a neighborhood of $\Arc(\overline{\lambda_1},\lambda_1)$,
furthermore by \eqref{derivative1} it satisfies
\begin{equation}
	\label{derivative2}
	\mathrm{Re}\, R_\lambda>\frac{ b ^2(d-1)-(d+1)}{2 b }>-1.
\end{equation}
\begin{lemma}
Let $1<b<\frac{d+1}{d-1}$. Then for every $\lambda\in \Arc (1,\lambda_1)$ we have $\mathrm{Im}\, R_\lambda>0$.
\end{lemma}
\begin{proof}
A simple calculation shows that $1$ is an attracting fixed point for $f_\lambda$ if and only if $\lambda=1$. This fact
together with \eqref{derivative2} imply that $R_1 = 1$ and that $R_\lambda \neq {\pm 1}$ as $\lambda\in \Arc(1,\lambda_1)$.
If we differentiate both sides of the equation $R_\lambda = f_\lambda(R_\lambda)$ with respect to $\lambda$, and then we
evaluate at $\lambda = 1$, we obtain
$$
	\partial_\lambda R_\lambda|_{\lambda=1}
		=
	\frac{R_\lambda}{\lambda(1-f'_\lambda(R_\lambda))}\Big\vert_{\lambda=1}
		=
	\frac{1}{1-d\frac{1- b }{1+ b }}
		>
	0.
$$
Therefore for $\lambda \in \Arc(1,\lambda_1)$ sufficiently close to $1$ the point $R_\lambda$ lies in the upper half plane.
However as $\lambda$ varies within $\Arc(1,\lambda_1)$, the point $R_\lambda$ moves on $\mathbb S^1$ without intersecting
$\{-1, 1\}$. Therefore $\mathrm{Im}\,R_\lambda>0$ on the whole $\Arc(1,\lambda_1)$.
\end{proof}

\begin{proof}[Proof of Theorem \ref{l0l1theorem}]
Let $z_\lambda,w_\lambda \in \mathbb S^1$ so that $I_\lambda = \Arc(z_\lambda,w_\lambda)$. Since the map
$f_\lambda: I_\lambda \rightarrow I_\lambda$ is an orientation reversing bijection, we have
$$
	f_\lambda(z_\lambda) = w_\lambda,
		\qquad
	f_\lambda(w_\lambda)=z_\lambda,
$$
showing that $z_\lambda, w_\lambda$ are fixed points for $f^2_\lambda$. The Fatou set is connected, therefore there can
be only one attracting or parabolic fixed point for $f_\lambda^2$, which is $R_\lambda$. This shows that the cycle
$z_\lambda,w_\lambda$ is repelling. By the implicit function theorem the points $z_\lambda, w_\lambda$ move holomorphically
and without collisions on some neighborhood $U \supset \Arc(\lambda_0,\lambda_1)$.

By the previous lemma and \eqref{derivative2} we have
$$
R_\lambda \in \left\{x+iy\,\big|\,y>0,\,x>\frac{ b ^2(d-1)-(d+1)}{2 b }>-1\right\}.
$$

Suppose now that for some $\lambda \in \Arc(\lambda_0,\lambda_1)$ we have $z_\lambda\in \Arc(1,R_\lambda)$. By \eqref{derivative1}
the map $f_\lambda$ is a contraction on $\Arc[1,R_\lambda]$. As the point $z$ moves counterclockwise on $\Arc[1,R_\lambda]$, its
image $f_\lambda(z)$ moves clockwise on $\mathbb S^1$ starting at $\lambda$ and ending at $R_\lambda$. Since $f_\lambda$ is a contraction
and $\mathrm{Im} R_\lambda>0$, this is possible only if $f_\lambda: \Arc[1,R_\lambda] \to \Arc[R_\lambda,\lambda]$ is an orientation reversing
bijection. We conclude that $w_\lambda = f_\lambda(z_\lambda) \in  \Arc(R_\lambda,\lambda)$ and thus that
$I_\lambda \Subset \Arc(1,\lambda)$.

If we differentiate both sides of $z_\lambda = f_\lambda^2(z_\lambda)$ we obtain the equation
\begin{equation}
	\label{450}
	\partial_\lambda z_\lambda\,{\left(1-(f_\lambda^2)'(z_\lambda)\right)}
		=
	\lambda^{-1}{\left(z_\lambda+f_\lambda'(w_\lambda)w_\lambda\right)}.
\end{equation}
Since $I_{\lambda_0} = \Arc(1,\lambda)$ it follows that $z_{\lambda_0} = 1$ is a repelling fixed point for $f^2_{\lambda_0}$,
furthermore since $|f_{\lambda_0}'(1)| < 1$ we must have $|f'_{\lambda_0}(\lambda_0)|>1$ and $|(f^2_{\lambda_0})'(1)|>1$.
If we evaluate the expression above at $\lambda = \lambda_0$ we obtain that
$$
	\partial_\lambda z_\lambda|_{\lambda = \lambda_0}
		=
	\frac{1}{\lambda_0}\frac{1-|f'_{\lambda_0}(\lambda_0)|}{1-|(f^2_{\lambda_0})'(1)|}
		=
	\frac{C}{\lambda_0},
$$
for some positive constant $C>0$. If we write $\lambda(\varepsilon) = \lambda_0e^{i\varepsilon}$ then we obtain that
$$
	z_{\lambda(\varepsilon)} = 1+iC\varepsilon+O(\varepsilon^2)
$$
therefore as $\lambda \in \Arc(\lambda_0,\lambda_1)$ is sufficiently close to $\lambda_0$, we must have
$z_\lambda \in \Arc(1,R_\lambda)$ and thus that $I_\lambda \Subset \Arc(1,\lambda)$.

This also proves that the point $z_\lambda$ moves counterclockwise as $\lambda$ is close to $\lambda_0$.
We will show that $z_\lambda$ moves counterclockwise on the whole arc between $\lambda_0$ and $\lambda_1$. Assume
otherwise, then there is some $\mu \in \Arc(\lambda_0,\lambda_1)$ such that
$$
	0
		=
	\partial_\lambda z_\lambda|_{\lambda=\mu}
		=
	\frac{1}{\mu}\left(z_\mu+f'_\mu(w_\mu)w_\mu\right).
$$
Note that, since $z_\lambda \neq R_\lambda$ for any $\lambda \in \Arc(\lambda_0,\lambda_1)$, it follows that
$\mu \in \Arc(1, R_\mu)$. As a result we must have $|f_{\mu}(z_\mu)|<1$, and therefore
$$
	|f'_\mu(z_\mu)|\cdot|f'_\mu(w_\mu)|
		=
	|f'_\mu(z_\mu)|\cdot\left|\frac{z_\mu}{w_\mu}\right|
		<
	1,
$$
which contradicts the fact that $z_\mu$ is a repelling fixed point of $f^2_\mu$.
This shows that $z_\lambda \in \Arc(1,R_\lambda)$ for every $\lambda \in \Arc(\lambda_0,\lambda_1)$ and therefore
$I_\lambda \Subset \Arc(1,\lambda)$, concluding the proof of the proposition.
\end{proof}



Recall that for $b > 1$ the point $-1$ is a hyperbolic parameter and that $\Lambda^{hyp}$ is an open and closed subset of
$\Lambda^{stb}$. Therefore the connected component $\Lambda^{hyp}_{-1}$ is a connected component of $\Lambda^{stb}$. The
same is clearly true for $\Lambda^{hyp}_1$ when $1$ is a hyperbolic parameter.
\begin{lemma}
\label{nonconstant}
For $b>1$ the component $\Lambda^{hyp}_{-1}$ is not exceptional.
For $1 < b < \frac{d+1}{d-1}$ the component $\Lambda^{hyp}_1$ is not exceptional.
\end{lemma}
\begin{proof}
When $1<b<\frac{d+1}{d-1}$ the component $\Lambda^{hyp}_1$ is not exceptional since $1$ is an attracting fixed point for the map $f_1$.

Suppose now that $\Lambda^{hyp}_{-1}$ is exceptional for some $b > 1$ and let $\varphi_\mu$ be the holomorphic motion
of $J_{-1}$ over $(\Lambda^{hyp}_{-1},-1)$. Since the holomorphic motion respects the dynamics we obtain that for every
$\mu \in \Lambda^{hyp}_{-1}$
\begin{equation}
\label{ilbuongermano}
f^n_\mu(1)=f^n_\mu\circ\varphi_\mu(1) = \varphi_\mu\circ f^n_{-1}(1).
\end{equation}
This shows that when the degree $d$ is even the function $f_\mu$ maps $1$ to a fixed point of $f_\mu$,
proving that $f^2_\mu(1) = f_\mu(1)$. On the other hand, when $d$ is odd the point $1$ is periodic with period two,
and therefore $f^2_\mu(1) = 1$. Once the values of $b$ and $d$ are fixed there are only finitely
many values of $\mu \in \Lambda^{hyp}_{-1}$ for which equation \eqref{ilbuongermano} is satisfied, giving a contradiction.
\end{proof}

\begin{corollary}
\label{corollary2comp}
Suppose that $b>1$ and $U=\Lambda^{hyp}_{-1}$, or alternatively that $1<b<\frac{d+1}{d-1}$ and $U=\Lambda^{hyp}_{1}$. Then the set of accumulation points of $\mathcal{Z}_{\mathcal{C}_{d+1}}$ in $U$ equals $\mathcal{J}_U$. Moreover,
if the degree $d$ is even then
\begin{equation}
\label{seconda}
\mathcal J_U=\overline{\mathcal Z_{\mathcal C_{d+1}}} \cap U.
\end{equation}
\end{corollary}
\begin{proof}
We first prove the statement for general degrees $d$. By Lemmas \ref{closureZ} and \ref{normalfamily} we obtain the inclusion
$$
	\mathcal J_U \cup (\mathcal Z_{\mathcal C_{d+1}} \cap U)\subset \overline{\mathcal Z_{\mathcal C_{d+1}}} \cap U.
$$
Therefore it suffices to show that every $\lambda \in \mathcal F_U$ is either an isolated point of $\mathcal Z_{\mathcal C_{d+1}}$
or is not contained in $\overline{\mathcal Z_{\mathcal C_{d+1}}}$. The map $f_{\lambda}$ is hyperbolic, therefore the orbit of $1$ converges
to an attracting periodic point $Q_\lambda$ of period $N$. By Proposition \ref{generaljulia} when $U = \Lambda^{hyp}_1$ the Fatou set $F_\lambda$
is connected and $N=1$. Similarly, when $U=\Lambda^{hyp}_{-1}$, the set $F_\lambda$ is the union of two distinct connected components,
and $N=2$.

The parameter $\lambda$ is passive, therefore $f^{2n+k}_\mu(1)\to f^k(Q_\mu)$ uniformly on some small ball $ B(\lambda,\varepsilon)$,
where $k=0,1$ and $Q_\mu$ is the holomorphic continuation of the periodic point $Q_\lambda$ . The point $-1$ cannot be an attracting periodic
point of order $1$ or $2$, thus $Q_{\mu},f_\mu(Q_\mu)\neq -1$. We conclude that whenever $n$ is sufficiently large the point $f^n_\mu(1)$ is
bounded away from $-1$. Therefore the intersection $B(\lambda,\varepsilon) \cap \mathcal Z_U$ only contains isolated points.

Now suppose that $d$ is even and let $\lambda\in \mathcal Z_{\mathcal C_{d+1}}\cap U$. Let $N > 0$ be the first integer so that
$f^N_{\lambda}(1) = -1$. Since $d$ is even the point $-1$ is periodic with period $N$. If $N > 2$ the point $-1$ is
a repelling periodic point, since attracting fixed points in $U$ have period $1$ or $2$. On the other hand $-1$ cannot be an attracting periodic
point with order $1$ or $2$. This shows that $\mathcal Z_{\mathcal C_{d+1}} \cap U \subset\mathcal J_U$, which proves \eqref{seconda}.
\end{proof}

\begin{proposition}
\label{cantorSet}
Let $1 < b < \frac{d+1}{d-1}$. Then the set $\mathcal J_{\Lambda^{hyp}_1}$ is a Cantor set,
with respect to the intrinsic topology of $\Lambda^{hyp}_1$.
\end{proposition}
\begin{proof}
Proposition~\ref{perfect} states that $\mathcal J_{\Lambda^{hyp}_1}$ is a perfect set.
Therefore we only need to show that every connected component $K$ of this set consists
of a single point. Let $\varphi_\mu$ be the holomorphic motion of $J_1$ over $(\Lambda^{hyp}_1,1)$.
By Theorem \ref{holomorphicmotionthm} the map $\varphi^{-1}_\mu(1)$ is continuous and sends
$K$ inside $J_1$. Since $J_1$ is a Cantor set we conclude that $\varphi^{-1}_\mu(1)$ is constant on $K$,
and therefore that $\varphi_\mu(c) = 1$ for some $c\in \widehat{\mathbb C}$ and every $\mu\in K$.

If $K$ contains more then one point then by the identity principle we must have $\varphi_\mu(c) = 1$
for all $\mu\in U$, and in particular $c = \varphi_1(c)=1$, thus showing that $\Lambda^{hyp}_1$ is an
exceptional component, contradicting Lemma \ref{nonconstant}. We conclude that $K$ is a single point.
\end{proof}
Combining the previous proposition with Corollary~\ref{corollary2comp} we conclude the proof of claim $(2)$
in Theorem \ref{thm:main}.

\end{section}


\section{Restriction to the unit circle}

Throughout this section it will be assumed that $b>1$.

\begin{lemma}\label{reppelingparnear}
Let $\lambda\in\mathbb S^1\cap \Lambda^{hyp}$, and let $\varphi_\mu$ be the holomorphic motion of $J_\lambda$ over $(\Lambda^{hyp}_\lambda,\lambda)$. Assume that $1\in J_\lambda$ and that one of the two following conditions is satisfied:
\begin{enumerate}
\item the partial derivative $\partial_\mu \varphi_\mu(1)|_{\mu=\lambda}\neq 0$,
\item there exist sequences $z_n,w_n\in J_{\lambda}$ both converging to $1$ and satisfying $1\in \Arc(z_n,w_n)$.
\end{enumerate}
Then there exists a sequence $\lambda_n\in\mathbb S^1$ converging to $\lambda$ so that $1$ is a repelling periodic point for each map $f_{\lambda_n}$.
\end{lemma}
\begin{proof}
Given $\lambda\in \mathbb S^1$ the Julia set $J_\lambda$ is contained in the unit circle. Therefore for every $\varepsilon>0$ there exists $\delta>0$ so that the map $\mu\mapsto\varphi_\mu(1)$ sends $I_\delta=B(\lambda,\delta)\cap\mathbb S^1$ inside a relatively compact subset of $I_\varepsilon=B(1,\varepsilon)\cap \mathbb S^1$. By continuity of the holomorphic motion we may assume that the same is true for the map $\mu\mapsto \varphi_\mu(z)$ whenever $z\in J_\lambda$ is sufficiently close to $1$. The map $\varphi_\mu(z):I_\delta\rightarrow I_\varepsilon$ can be interpreted as a map between intervals. We will denote its graph as $\Gamma(z)\subset I_\delta\times I_\varepsilon$.


Assume first that condition $(1)$ holds. Let $z_n\in J_{\lambda}$ be a sequence of repelling periodic points so that $z_n\to 1$ and $z_n\neq 1$. Since $\partial_\mu \varphi_\mu(1)|_{\mu=\lambda}\neq 0$ the graph $\Gamma(1)$ intersects the line $w=1$ transversally at the point $(\lambda,1)$. By Theorem \ref{holomorphicmotionthm}, when $n$ is sufficiently large the graph $\Gamma(z_n)$ is uniformly close to $\Gamma(1)$ and therefore it intersects the line $w=1$ in $(\lambda_n,1)$ for some $\lambda_n\in I_\delta\setminus\{\lambda\}$ close to $\lambda$. It follows that $\lambda_n\to\lambda$ and that $1=\varphi_{\lambda_n}(z_n)$ is a repelling point for $f_{\lambda_n}$


Assume now that condition $(2)$ holds. Since repelling fixed points are dense in the Julia set, we may assume from the beginning that the $z_n$ and $w_n$ are repelling fixed points for $f_\lambda$. When $n$ is sufficiently large the graphs $\Gamma(z_n)$ and $\Gamma(w_n)$ are both close to $\Gamma(1)$.  Furthermore, since the map $z\mapsto \varphi_\mu(z)$ is injective, we conclude that
$$
\varphi_\mu(1)\in \Arc\big(\varphi_{\mu}(z_n),\varphi_{\mu}(w_n)\big),\qquad\forall \mu\in I_\delta,
$$
meaning that the graph $\Gamma(1)$ lies in between $\Gamma(z_n)$ and $\Gamma(w_n)$.

It follows that when $n$ is sufficiently large there exists $\lambda_n\in I_\delta\setminus \{\lambda\}$ close to $\lambda$ so that either $\varphi_{\lambda_n}(z_n)=1$ or $\varphi_{\lambda_n}(w_n)=1$.  As in the previous case we obtain that $1$ is a repelling fixed point for $f_{\lambda_n}$ and that $\lambda_n\to\lambda$, concluding the proof of the proposition.
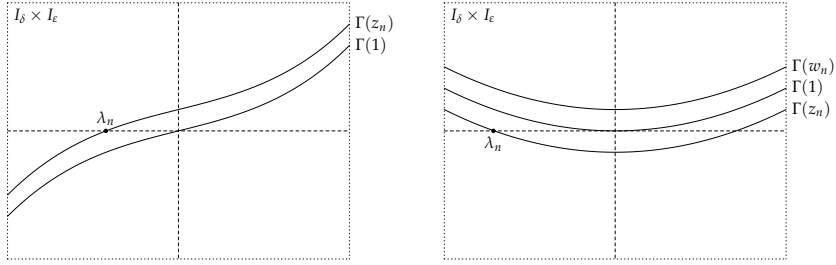
\begin{figure}[ht]
\center
\large
\resizebox{0.45\textwidth}{!}{
\begin{tikzpicture}[line cap=round,line join=round,>=triangle 45,x=1.0cm,y=1.0cm]
\clip(-4.1,-3.1) rectangle (5.7,3.1);
\draw[smooth,samples=100,domain=-4.0:4.0] plot(\x,{((\x)/4)^3+(\x)/4});
\draw [dash pattern=on 2pt off 2pt] (-4,0)-- (4,0);
\draw [dash pattern=on 2pt off 2pt] (0,3)-- (0,-3);
\draw[smooth,samples=100,domain=-4.0:4.0] plot(\x,{((\x)/4)^3+(\x)/4+0.5});
\draw [dotted] (-4,3)-- (-4,-3);
\draw [dotted] (-4,-3)-- (4,-3);
\draw [dotted] (-4,3)-- (4,3);
\draw [dotted] (4,3)-- (4,-3);
\node[below right] at (-4,3){$I_\delta\times I_\varepsilon$};
\node[right] at (4,2.5){$\Gamma(z_n)$};
\node[right] at (4,2){$\Gamma(1)$};
\fill (-1.7,0) circle (1.5pt);
\node[above] at (-1.7,0){$\lambda_n$};
\end{tikzpicture}
}
\resizebox{0.45\textwidth}{!}{
\begin{tikzpicture}[line cap=round,line join=round,>=triangle 45,x=1.0cm,y=1.0cm]
\clip(-4.1,-3.1) rectangle (5.7,3.1);
\draw[smooth,samples=100,domain=-4.0:4.0] plot(\x,{((\x)/4)^2});
\draw [dash pattern=on 2pt off 2pt] (-4,0)-- (4,0);
\draw [dash pattern=on 2pt off 2pt] (0,3)-- (0,-3);
\draw[smooth,samples=100,domain=-4.0:4.0] plot(\x,{((\x)/4)^2+0.5});
\draw[smooth,samples=100,domain=-4.0:4.0] plot(\x,{((\x)/4)^2-0.5});
\draw [dotted] (-4,3)-- (-4,-3);
\draw [dotted] (-4,-3)-- (4,-3);
\draw [dotted] (-4,3)-- (4,3);
\draw [dotted] (4,3)-- (4,-3);
\node[below right] at (-4,3){$I_\delta\times I_\varepsilon$};
\node[right] at (4,1){$\Gamma(1)$};
\node[right] at (4,1.5){$\Gamma(w_n)$};
\node[right] at (4,0.5){$\Gamma(z_n)$};
\fill (-2.85,0) circle (1.5pt);
\node[below] at (-2.85,0){$\lambda_n$};
\end{tikzpicture}
}
\caption{The graphs of the holomorphic motions respectively in case $(1)$ and $(2)$.}
\end{figure}
\end{proof}

\begin{proposition}\label{badparrepelling}
Let $\lambda\in\mathbb S^1$ be so that $1$ is a repelling periodic point for $f_\lambda$. Then $\lambda\in \overline{\mathcal Z_{\mathcal C_{d+1}}\cap \mathbb S^1}$.
\end{proposition}
This proposition is proved after Lemma~\ref{expansion0}.
We claim that is enough to prove the statement for hyperbolic parameters. When $b>\frac{d+1}{d-1}$ by Proposition \ref{noneutral} all points in the circle are hyperbolic, and the claim is certainly true. Assume instead that $1<b\le \frac{d+1}{d-1}$ and that the Proposition holds for hyperbolic parameters. Given $\lambda\in \Lambda^{hyp}_{-1}\cap\mathbb S^1$ by Corollary \ref{generaljulia} the Julia set $J_\lambda$ coincides with the unit circle, and by Lemma \ref{reppelingparnear} we may find parameters in $\Lambda^{hyp}_{-1}\cap\mathbb S^1$ arbitrarily close to $\lambda$ for which $1$ is a repelling periodic point. It follows that $\lambda\in \overline{\mathcal Z_{\mathcal C_{d+1}}\cap \mathbb S^1}$ and that
$$
\overline{\Lambda^{hyp}_{-1}\cap\mathbb S^1}\subseteq \overline{\mathcal Z_{\mathcal C_{d+1}}\cap \mathbb S^1}.
$$
This shows that the proposition holds also when $\lambda\in\mathbb S^1$, proving the claim (by Proposition \ref{noneutral}, non-hyperbolic point on the circle are in $\overline{\Lambda^{hyp}_{-1}\cap\mathbb S^1}$).

From now on we will fix $\lambda\in\mathbb S^1\cap\Lambda^{hyp}$ so that $1$ is a repelling point for $f_\lambda$. Given such $\lambda$ we will write $N$ for the period of the point $1$, and $\varphi_\mu$ for the holomorphic motion of $J_\lambda$ over $(\Lambda^{hyp}_\lambda,\lambda)$.
\begin{lemma}\label{period3}
Suppose that $b\ge \frac{d+1}{d-1}$ or that $N\ge 3$. Then there exist sequences $z_n,w_n\in J_{\lambda}$ both converging to $1$ and satisfying $1\in \Arc(z_n,w_n)$.
\end{lemma}
\begin{proof}
Suppose first that $b\ge \frac{d+1}{d-1}$ or that $1<b<\frac{d+1}{d-1}$ and $\lambda\in \Arc(\lambda_1,\overline\lambda_1) $. In this case $J_{\lambda}=\mathbb S^1$, and the result follows immediately.

When instead $1<b<\frac{d+1}{d-1}$ and $\lambda\in \Arc (\overline\lambda_1,\lambda_1)$, then by Proposition \ref{generaljulia} the Julia set $J_{\lambda}\subset \mathbb S^1$ is a Cantor set. Suppose for the purpose of a contradiction that the Lemma were false, so that $1\in\partial \widehat I$, where $\widehat I$ is a connected component of $F_\lambda\cap\mathbb S^1$. The connected components of $F_\lambda\cap \mathbb S^1$ are open arcs which are mapped one to another by $f_{\lambda}$ and are eventually mapped into the invariant arc $I_\lambda$ containing the unique attracting fixed point $R_{\lambda}$.

If $1\in\partial \widehat I$ it follows that for some $n>0$ we must have $f^n_\lambda(1)\in\partial I_\lambda$. But $I_\lambda$ is invariant and we are assuming that $1$ is periodic, therefore $1\in\partial I_\lambda$. However this is not possible when $N\ge 3$ since the boundary points of $I_\lambda$ are repelling periodic points with period $2$, giving a contradiction.
\end{proof}

The point $1$ is a fixed point if and only if $\lambda=1$. If $1<b<\frac{d+1}{d-1}$ then the point $1$ is an attracting fixed point for $f_1$, thus if $1$ is a repelling periodic point for $f_\lambda$ we must have $N\ge 2$.

\begin{lemma}\label{period2}
Suppose that $1<b<\frac{d+1}{d-1}$ and that $N=2$. Then $\partial_\mu \varphi_\mu(1)|_{\mu=\lambda}\neq 0$.
\end{lemma}
\begin{proof}
Suppose instead that $\partial_\mu \varphi_\mu(1)|_{\mu=\lambda}=0$ and write $f_\lambda(z)=\lambda g(z)$. Since the point $\varphi_\mu(1)$ is a repelling periodic point of period $2$ for $f_\mu$ it follows that
\begin{align*}
0&=\partial_\mu \Big(f_\mu^2\circ\varphi_\mu(1)\Big)\Big|_{\mu=\lambda}\\
&=g(\lambda)+f'_\lambda(\lambda).
\end{align*}
For $\lambda\in\mathbb S^1$ we then have $|f_{\lambda}'(\lambda)|=|g(\lambda)|=1$, and therefore $|(f^2_{\lambda})'(1)|=|f'_{\lambda}(1)|<1$, contradicting the fact that $1$ is a repelling fixed point with period $2$.
\end{proof}

Since $\lambda$ is a hyperbolic parameter, there exists an integer $j\ge 1$, an open neighborhood $U\supset J_{\lambda}$ and $\kappa>1$ so that whenever $z,w\in U$ are sufficiently close we have
$$
| f^{jN}_{\lambda}(z)-f^{jN}_{\lambda}(w)|\ge \kappa| z-w|.
$$
Given $\delta>0$ sufficiently small we may further assume that the same is true for $f_\mu$ when we take $\mu\in I_\delta=B(\lambda,\delta)\cap \mathbb S^1$. Since the map $\mu\mapsto J_\mu$ is continuous with respect to the Hausdorff distance, we may further assume that $J_\mu\subset U$ for every $\mu\in I_\delta$. We therefore obtain the following:
\begin{lemma}\label{expansion0}
There exists $\varepsilon>0$ so that for every $\mu\in I_\delta$ and $z,w\in J_\mu$ distinct we can find $k\ge 0$ so that
$$
| f^{kN}_{\mu}(z)-f^{kN}_{\mu}(w)|\ge 2\varepsilon.
$$
\end{lemma}

\begin{proof}[Proof of Proposition \ref{badparrepelling}]
Let $\lambda\in\Lambda^{hyp}\cap\mathbb S^1$ be a parameter for which $1$ is a repelling periodic point of period $N\ge 2$. We will assume first that $b\ge \frac{d+1}{d-1}$ or that $N\ge 3$.

By Lemma \ref{period3} there exist two sequences in $J_\lambda$ converging to $1$; one contained in the upper half plane and one in the lower half plane. Since the backward images of the point $-1$ accumulates on the Julia set $J_\lambda$, and thus on every point in such sequences, we may find two preimages $\alpha_\lambda,\beta_\lambda$ of the point $-1$ contained in $I_\varepsilon=B(1,\varepsilon)\cap\mathbb S^1$ so that $1\in \Arc(\alpha_\lambda,\beta_\lambda)$. Write $M_1,M_2>0$ for the two positive integers so that $f^{M_1}(\alpha_\lambda)=f^{M_2}(\beta_\lambda)=-1$.

Take $\delta>0$ and write $I_\delta=B(\lambda,\delta)\cap\mathbb S^1$. Then if $\delta$ is sufficiently small we have $\varphi_\mu(1):I_\delta\rightarrow I_\varepsilon$ and we can find two continuous functions $\alpha_\mu, \beta_\mu : I_\delta\rightarrow I_\varepsilon$ so that $f^{M_1}(\alpha_\mu)=f^{M_2}(\beta_\mu)=-1$  and
$$
\varphi_\mu(1)\in \Arc (\alpha_\mu,\beta_\mu)\qquad\forall\mu\in I_\delta.
$$

Lemma \ref{period3} shows that condition $(2)$ in Lemma \ref{reppelingparnear} is satisfied. It follows that there exists $\lambda'\in I_\delta\setminus\{\lambda\}$ arbitrarily close to $\lambda$ so that $1$ is also a repelling periodic point for $f_{\lambda'}$, and therefore $1\in J_{\lambda'}$. Furthermore by Lemma \ref{nonconstant} the holomorphic map $\varphi_\mu(1)$ is not constant, therefore we may choose $\lambda'$ so that $1\neq \varphi_{\lambda'}(1)$.

By Lemma \ref{expansion0} we may find $k$ so that $| f^{kN}_{\lambda'}(\varphi_{\lambda'}(1))-f^{kN}_{\lambda'}(1)|\ge 2\varepsilon$, and since $\varphi_{\lambda'}(1)\in I_\varepsilon$ is a periodic point for $f_{\lambda'}$ of period $N$ we conclude that
$$
|f^{kN}_{\lambda'}(1)-1|\ge | f^{kN}_{\lambda'}(1)-\varphi_{\lambda'}(1)|-| \varphi_{\lambda'}(1)-1|\ge \varepsilon,
$$
showing that $f^{kN}_{\lambda'}(1)\in \mathbb S^1\setminus I_\varepsilon$.

The map $\mu\mapsto f^{kN}_\mu(1)$ is continuous and we have
$$f^{kN}_\lambda(1)=1\in \Arc(\alpha_\lambda,\beta_\lambda),\qquad f^{kN}_{\lambda'}(1)\not\in \Arc(\alpha_{\lambda'},\beta_{\lambda'}).$$ Therefore we may conclude that there exists $\lambda''\in \Arc(\lambda,\lambda')$ so that either $f^{kN}_{\lambda''}(1)=\alpha_{\lambda''}$ or $f^{kN}_{\lambda''}(1)=\beta_{\lambda''}$. Since $\alpha_{\lambda''}$ and $\beta_{\lambda''}$ are preimages of the point $-1$ we conclude in both cases that $\lambda''\in\mathcal Z_{\mathcal C_{d+1}}\cap\mathbb S^1$. Since $\lambda'$, and thus $\lambda''$, can be chosen arbitrarily close to $\lambda$, we must have $\lambda\in \overline{\mathcal Z_{\mathcal C_{d+1}}\cap\mathbb S^1}$.



Suppose now that $1<b<\frac{d+1}{d-1}$ and that $N=2$. We notice that once $ b $ and $d$ are fixed, this can happen only for finitely many values of $\lambda$. Combining Lemma \ref{period2} and Lemma \ref{reppelingparnear} we can find $\lambda'\in \Lambda^{hyp}\cap\mathbb S^1$ arbitrarily close to $\lambda$ for which $1$ is a repelling fixed point for $f_{\lambda'}$ with period greater than $3$. Since the proposition holds for $\lambda'$ it must hold for $\lambda$ as well, concluding the proof of the proposition.
\begin{figure}[ht]
\center
\large
\resizebox{0.7\textwidth}{!}{
\begin{tikzpicture}[line cap=round,line join=round,>=triangle 45,x=1.0cm,y=1.0cm]
\clip(-5.2,-4) rectangle (6.7,5);
\draw[smooth,samples=100,domain=-5.0:5.0] plot(\x,{((\x)/4)^2});
\draw [dotted] (-5,0)-- (5,0);
\draw [dotted] (0,3)-- (0,-3);
\draw[smooth,samples=100,domain=-5.0:5.0] plot(\x,{((\x)/4)^2+0.5});
\draw[smooth,samples=100,domain=-5.0:5.0] plot(\x,{((\x)/4)^2-0.5});
\draw[smooth,red,samples=100,domain=-4.7622:4.7622] plot(\x,{(\x/3)^3});
\draw [dotted] (-5,3)-- (-5,-3);
\draw [dotted] (-5,-3)-- (5,-3);
\draw [dotted] (-5,3)-- (5,3);
\draw [dotted] (5,3)-- (5,-3);
\draw [dotted] (4.5,0)-- (4.5,3.375);
\draw [dotted] (3.096,0)-- (3.096,1.099);
\draw [loosely dotted] (5,4)-- (5,3);
\draw [loosely dotted] (0,4)-- (0,3);
\draw [loosely dotted] (-5,4)-- (-5,3);
\draw [loosely dotted] (5,-4)-- (5,-3);
\draw [loosely dotted] (0,-4)-- (0,-3);
\draw [loosely dotted] (-5,-4)-- (-5,-3);
\draw [loosely dotted] (-5,-4)-- (5,-4);
\draw [loosely dotted] (-5,4)-- (5,4);
\node[below right] at (-5,3){$I_\delta\times I_\varepsilon$};
\node[right] at (5,1.5){$\varphi_\mu(1)$};
\node[right] at (5,2){$\beta(\mu)$};
\node[right] at (5,1){$\alpha(\mu)$};
\node[red,above] at (4.7622,4){$f^{Nk}_\mu(1)$};
\fill (4.5,0) circle (1.5pt);
\fill[red] (4.5,3.375) circle (1.5pt);
\node[below] at (4.5,0){$\lambda'$};
\fill (3.096,0) circle (1.5pt);
\fill[red] (3.096,1.099) circle (1.5pt);
\node[below] at (3.096,0){$\lambda''$};
\end{tikzpicture}
}
\caption{The position of the points $\lambda'$ and $\lambda''$.}
\end{figure}
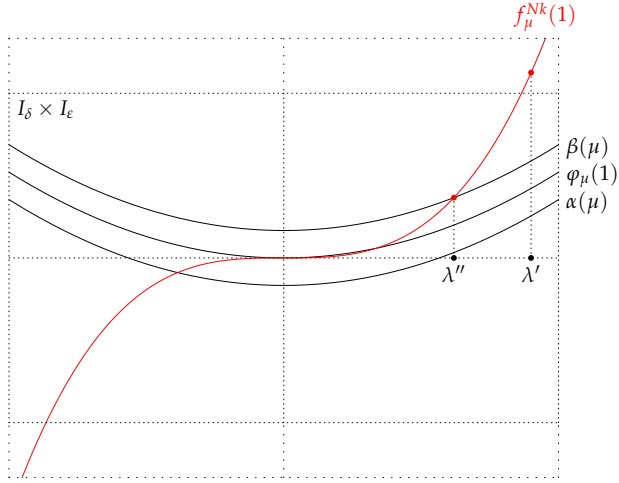
\end{proof}

\begin{proof}[Proof of claim $(1)$ in Theorem~\ref{thm:main}]

We already showed that Lemma~\ref{reppelingparnear} together with Proposition~\ref{badparrepelling} imply $\overline{\Lambda_{-1}^{hyp}\cap \mathbb{S}}\subset \overline{\mathcal{Z}_{\mathcal{C}_{d+1}}\cap\mathbb{S}}$. Therefore when $1<b<\frac{d+1}{d-1}$ by Proposition~\ref{noneutral} we have
\[
  \Arc[\lambda_1,\overline{\lambda_1}]\subset \overline{\mathcal{Z}_{\mathcal{C}_{d+1}}\cap\mathbb{S}}.
\]

If $\lambda=\lambda_0$ or $\lambda=\overline{\lambda_0}$, then $1$ is a repelling 2-cycle of $f_{\lambda_0}$, therefore by Proposition~\ref{badparrepelling} we have $\{\lambda_0,\overline{\lambda_0}\}\subset \overline{\mathcal{Z}_{\mathcal{C}_{d+1}}\cap\mathbb{S}}$.
\end{proof}

\section{Hyperbolic semigroups and expanding orbits}

All throughout this section we will assume that $d \in \mathbb N_{\ge 2}$ and $b \in (1,\frac{d+1}{d-1})$
are fixed.

\begin{definition}
Given $\lambda \in \hat{\mathbb C}$ we define the semigroup $H = \langle f_{1},\dots, f_{d}\rangle$
as the semigroup generated by the maps
$$
f_{k}(z) = \lambda\,\left(\frac{z + b}{b z + 1}\right)^k,\qquad k=1,\dots, d.
$$
We will write $F_{k}, J_{k}$ for the Fatou and Julia set of the map $f_{k}$ and $F_{H}, J_{H}$
for the Fatou and Julia set of the semigroup $H$.
\end{definition}

The following characterization of the semigroup $H$ will not be used later in the paper.

\begin{proposition}
For all but countably many $\lambda\in\widehat{\mathbb C}$ the semigroup $H$ is freely generated by $\{f_1,\dots,f_d\}$.
\end{proposition}
\begin{proof}
Let $S$ be the free group generated by $\{f_1,\dots,f_d\}$ and write $\Phi:S\rightarrow H$ for the homomorphism
$$
\Phi[f_{i_1}\cdots f_{i_k}]=f_{i_k}\circ \dots\circ f_{i_1}.
$$
Given a word $s=f_{i_1}\cdots f_{i_k}\in S$ the map $\Phi[s]$ is a rational map in both $z$ and $\lambda$. Its degrees with respect to $z$ and $\lambda$ are equal to
\begin{align*}
\deg_z (s)&=i_1\cdots i_k,\\
\deg_\lambda\,(s)&= 1+i_k+i_k\cdot i_{k-1}+\dots +i_k\cdots i_1.
\end{align*}
We notice that the map $\Phi[s](1)$ is also a rational map in $\lambda$ of degree
$$
\deg_\lambda'(s)=1+i_k+i_k\cdot i_{k-1}+\dots i_k\cdots i_2.
$$

The set $S$ is countable, therefore it is sufficient to show that for every $s_1,s_2\in S$ either $\Phi(s_1)=\Phi(s_2)$ for finitely many $\lambda$ or $s_1=s_2$.

Suppose instead that $s_1=f_{i_1}\cdots f_{i_k}$ and $s_2=f_{j_1}\cdots f_{j_l}$ satisfy $\Phi(s_1)=\Phi(s_2)$ for infinitely many $\lambda$. Using the identity principle, it is not hard to show that $\Phi(s_1)(z)=\Phi(s_2)(z)$ holds for all $\lambda\in\widehat{\mathbb C}$. This shows that the maps $\Phi(s_1)$ and $\Phi(s_2)$ coincide as rational maps in both variables $z$ and $\lambda$, and therefore that
$$
i_1=\frac{\deg_z(s_1)}{\deg_\lambda(s_1)-\deg'_\lambda(s_1)}=\frac{\deg_z(s_2)}{\deg_\lambda(s_2)-\deg'_\lambda(s_2)}=j_1.
$$
If we now write $s_1'=f_{i_2}\cdots f_{i_k}$ and $s_2'=f_{j_2}\cdots f_{j_l}$ we obtain that $\Phi(s'_1)=\Phi(s'_2)$ also holds for infinitely many $\lambda$ and therefore $i_2=j_2$. Iterating this procedure we obtain that $s_1=s_2$, concluding the proof of the proposition.
\end{proof}

Notice that for $k \le d$ we have $\frac{k+1}{k-1} \ge \frac{d+1}{d-1}$ and hence
$b \in (1,\frac{k+1}{k-1})$.

\begin{definition}[Notation]
In order to maintain the notation readable we are avoiding (where possible) the use of the subscript $\lambda$.
As an example, notice that we are writing $f_k$ instead of the more accurate $f_{k,\lambda}$.
The function $f_d$ will play an important role in the analysis of the semigroup $H$.
Therefore we will write $\lambda_0$ and $\lambda_1$ for the parameters obtained by Theorem \ref{l0l1theorem}
applied to the map $f_d$, and $I$ for the immediate attracting arc relative to the function $f_d$ (if it exists).
When it is necessary to distinguish between different values of the parameter $\lambda$ we will
write $f_{k,\lambda}$ and $I_{\lambda}$.
\end{definition}

We will write $\Omega$ for the set of all possible sequences with entries in $\{1,\dots,d\}$.
For every element $g\in H$ we can find $\omega\in\Omega$ and $n\in\mathbb N$ so that $g=f^n_\omega$ where we write
$$
f^n_{\omega} = f_{\omega_n} \circ \dots \circ f_{\omega_1}.
$$
For $0 \le m \le n$ we will further write
$$
f^{n,m}_\omega = f_{\omega_n} \circ \dots \circ f_{\omega_m}.
$$

\begin{definition}
We define the length of an element $g \in H$ as the minimum integer $n$ for which
$g = f^n_\omega$ for some $\omega \in \Omega$.
\end{definition}


We proved in section 4 that for any $\lambda$ in the closed arc $\Arc[\lambda_1,-1]$
there exists a (rooted) Cayley tree $T_n$ of degree $d$ and a parameter $\lambda' \in \mathbb S^1$ arbitrarily
close to $\lambda$ so that
$$
Z_{T_n}(\lambda', b )=0.
$$
Recall from \cite[Theorem B]{PetersRegts2018} on the other hand that for any $\lambda\in [1,\lambda_0)$,
any $r \ge 0$ and any graph $G$ of degree $d$ we have
$$
Z_G(r\cdot\lambda, b)\neq 0.
$$

The situation for the arc $\Arc(\lambda_0,\lambda_1)$ appears to be more complicated.
If we only consider Cayley trees of degree $d$ it is possible to show that indeed there exist zeros
of the partition function of some tree $T_n$ on the arc $\Arc(\lambda_0,\lambda_1)$. However as shown
in section \ref{sec:dynamicsf}, these zeros form a nowhere dense subset of the arc.
Our purpose is to show that zeros of the partition for general bounded degree graphs are dense in
$\Arc[\lambda_0,\lambda_3]$, where $\lambda_3$ is a parameter on the unit circle close to $\lambda_0$.
In order to do so we will consider the class of rooted spherically symmetric trees of bounded degree
for which, according to Lemma~\ref{lemmatwo}, the zero parameters can be understood by studying the
semigroup dynamics of $H$.

In what follows we will study the semigroup dynamics (mostly) under the assumption that $\lambda \in
\mathbb S^1$. Under these assumptions for any $g \in H_\lambda$ we have
$g(\hat{\mathbb C} \setminus \mathbb S^1) = \hat{\mathbb C}\setminus\mathbb S^1$, and therefore that
$$
\hat{\mathbb C} \setminus \mathbb S^1 \subset F_{H_\lambda},
\qquad
J_{ H_\lambda} \subset \mathbb S^1.
$$

\subsection{Hyperbolicity of the semigroup}

\begin{lemma}
\label{doublemap}
There exists $\lambda_2 \in \Arc(\lambda_0,\lambda_1)$ so that for every $\lambda\in
\Arc[\lambda_0,\lambda_2)$ and every $k\in\{1,\dots,d-1\}$ we have
\begin{equation}
\label{goodinclusion}
f_1\circ f_k\left(I\right)\Subset I.
\end{equation}
\end{lemma}

\begin{proof}
Assume first that $\lambda = \lambda_0$ and recall that by Theorem \ref{l0l1theorem} we
have $I = \Arc(1,\lambda_0)$. The M\"obius transformation $\gamma(z) = (z + b)/(b z + 1)$ is
a bijection of the unit circle into itself which revers the orientation, and $f_k(z) = \lambda_0\gamma(z)^k$.
The map $f_d: I \rightarrow I$ is an orientation reversing bijection and therefore satisfies $f_d(I) = I$. It
follows immediately that, if we write $\ell$ for the length of the arc $I$, the image $f_k(I)$ is an arc of
length $\ell^{k/d}<\ell$, therefore the map $f_k:I\rightarrow \mathbb S^1$ is not surjective.

Notice that when the point $z$ move counterclockwise on the arc $I= \Arc(1,\lambda_0)$ starting at the point $1$,
its image $f_k(z)$ moves clockwise on the unit circle, starting at $\lambda_0 = f_k(1)$, until it reaches
$f_k(\lambda_0)$. In principle it is possible that $f_k(z)$ rotates once or more times around the circle,
however since $f_k$ is not surjective on $I$ this does not happen. We conclude that $f_k :
I \rightarrow \Arc(f_k(\lambda_0),\lambda_0)$ is also an orientation reversing bijection, and since the length of the arc
$f_k(I)$ is less than the length of $I$, we must have $f_k(\lambda_0)\in I$. If we now compose with the map $f_1$ we
find that
$$
f_1\circ f_k(I)=f_1\big(\Arc(f_k(\lambda_0),\lambda_0)\big) = \Arc \big(f_1(\lambda_0),f_1\circ f_k(\lambda_0)\big)
\Subset I.
$$
Since $I$ moves continuously in $\lambda$, the same holds sufficiently close $\lambda_0$,
which implies the existence of  $\lambda_2$.
\end{proof}

In the following we will denote by $\lambda_2$ the parameter with the maximal argument that satisfies
the requirements of the previous lemma.

\begin{definition}
We define the semigroup $\widehat{H} \subset H$ as the semigroup generated by the maps
$$
\widehat{H} = \langle \widehat f_1,\dots,\widehat f_d\rangle,
$$
where $\widehat f_d = f_d$ and $\widehat f_k = f_1\circ f_k$ for $k\in\{1,\dots,d-1\}$. We will
write $F_{\widehat {H}}, J_{\widehat {H}}$ for the Fatou and the Julia set of the semigroup
$\widehat {H}$.
\end{definition}

Since $F_{H} \subset F_{\widehat H}$, it follows that $\hat{\mathbb C} \setminus \mathbb S^1
\subset F_{\widehat H}$. Furthermore by the previous lemma the interval $I$ is invariant for every map in
$\widehat {H}$ for $\lambda \in \Arc[\lambda_0,\lambda_2)$, and therefore it is contained in
$F_{\widehat H}$, proving that for these $\lambda$
\begin{equation}
\label{Juliasetinclusion}
(\widehat{\mathbb C}\setminus\mathbb S^1) \cup  I\subset F_{\widehat H},
\qquad
J_{\widehat H} \subset \mathbb S^1 \setminus I,
\end{equation}
and therefore that the Fatou set $F_{\widehat H}$ is connected.

Similarly to the case of the semigroup $H$, given $\omega \in \Omega$ and $0 \le m \le n$ we will write
\begin{align*}
\widehat f^n_\omega 	&= \widehat f_{\omega_n}\circ \dots\circ \widehat f_{\omega_1},\\
\widehat f^{n,m}_\omega &= \widehat f_{\omega_n}\circ \dots\circ \widehat f_{\omega_m}.
\end{align*}
Also in this case every element of the semigroup can be written as $\widehat f^n_\omega$ for some $\omega\in\Omega$ and $n\in\mathbb N$.


By the previous lemma, given $\lambda\in \Arc[\lambda_0,\lambda_2)$ it follows that there exists
a closed arc $J \subset I$ so that $\widehat{f}_k(J) \Subset J$ for every $k=1,\dots, d$. Indeed, if we write
$I = \Arc(z,w)$, this property is satisfied by $J = \Arc[\zeta,\eta]$, where $\zeta$
is sufficiently close to $z$ and $\eta$ lies in $\Arc(f_d(\zeta), f_d^{-1}(\zeta))$ (where the preimage is taken
inside $I$). We may therefore define the family $\mathcal C$ and the closed arc $K$ as follows:
\begin{align}
\label{setK}
\begin{split}
\mathcal C 	&:=\{J\subset I\,|\, J\neq \emptyset
					\text{ closed arc such that }\widehat{f}_k(J)\Subset \mathrm{int}(J)\,\,\forall k=1,\dots,d\},\\
K 			&:=\bigcap_{J\in\mathcal C}J.
\end{split}
\end{align}
In this definition $\mathrm{int}(J)$ refers to the interior of $J$ with respect to the topology of the unit circle.

\begin{lemma}
\label{lemmaK}
Let $\lambda\in \Arc[\lambda_0,\lambda_2)$. Then the set $K \subset I$ is a non-empty closed arc which is
forward invariant under $\widehat H$. Furthermore the map $\lambda \mapsto K_\lambda$ is upper semi-continuous
for $\lambda\in  \Arc[\lambda_0,\lambda_2)$.
\end{lemma}
\begin{proof}
Write $\mathcal C_{\mathbb Q}$ for the subset of all intervals $J\in\mathcal C$ whose extrema are rational
angles. Notice that given $J \in \mathcal C$ there exists $J' \in \mathcal C_{\mathbb Q}$ for which
$J'\subset J$, and therefore
$$
K = \bigcap_{J \in \mathcal C_{\mathbb Q}} J.
$$

Given $J \in \mathcal C$ we have that $f_d(J) \Subset J$ and $J\subset I$, proving that the arc $J$ contains
the unique attracting fixed point of $f_d$. This shows that for any pair $J_1, J_2\in\mathcal C$ the intersection
$J' = J_1 \cap J_2$ is non empty. Furthermore one can show that $J'$ is again an element of $\mathcal C$. Similarly
given $J_1,J_2\in\mathcal C_{\mathbb Q}$ one has that $J_1 \cap J_2 \in \mathcal C_{\mathbb Q}$.

The set $\mathcal C_{\mathbb Q}$ is countable, hence we can enumerate its elements as
$\{J_k\}_{k \geq 1}$. If we write $L_k = J_1 \cap \dots \cap J_k$ it then follows that
$L_k\in\mathcal C_{\mathbb Q}$, that $L_{k+1} \subset L_k$ and that
$$
K = \bigcap_{k=1}^\infty L_k.
$$

This shows that $K$ is the intersection of a nested family of non-empty, compact and connected arcs,
and therefore that $K$ is a non-empty closed arc. Since every $L_k$ is forward invariant for any
$g \in \widehat H$, the same holds for $K$.

We will now write $\mathcal C_\mu, K_\mu$ in order to study parameter values close to $\lambda$.
On the other hand we will write $L_k$ for the set defined above for the parameter value $\lambda$.
For every positive integer $k$ we have $L_k \in \mathcal C_\mu$ for every $\mu\in \Arc[\lambda_0,\lambda_2)$
sufficiently close to $\lambda$, and for such parameters $\mu$ it then follows that $K_\mu \subset L_k$.
Since $L_k$ is a sequence of nested sets approximating $K_\lambda$ we conclude immediately that the map
$\mu \mapsto K_\mu$ is upper-semicontinuous at the point $\lambda$, concluding the proof of the lemma.
\end{proof}

\begin{proposition}
\label{convergencetoK}
Let $\lambda \in \Arc[\lambda_0,\lambda_2)$. Then every limit of a convergent sequence in
$\widehat H$ with divergent length is constant on $F_{\widehat H}$ and contained in $K$.
\end{proposition}

\begin{proof}
Let $g_k \in \widehat H$ be a sequence with divergent length that converges uniformly to
$g_\infty: F_{\widehat H} \to \overline{F}_{\widehat H} = \hat{\mathbb C}$. Choose sequences
$\omega_k \in \Omega$ and $n_k \in \mathbb N$ so that $g_k = \widehat f^{n_k}_{\omega_k}$, and notice
that since the sequence $g_k$ has divergent length we must have $n_k \to \infty$. Since $F_{\widehat H}
\cap \mathbb S^1 \neq \emptyset$, it follows that $F_{\widehat H}$ is connected. Write $\rho$ for the
hyperbolic metric of $F_{\widehat H}$.


The Fatou set $F_{\widehat H}$ is forward invariant with respect to each element in the semigroup.
In particular we have $\widehat{f}_1(F_{\widehat H})\subset F_{\widehat H}$. Since the closed arc
$K \subset F_{\widehat H}$ is forward invariant, it contains the unique attracting fixed point
$R_1$ of $\widehat f_1$ (which coincides with the attracting fixed point of $f_1$). Assume now that
$\widehat f_1(F_{\widehat H})=F_{\widehat H}$. Since $\widehat f_1$ is invertible, it then follows that
the Fatou set $F_{\widehat H}$ is completely invariant with respect to $\widehat f_1$, and therefore it
contains the whole attracting basin of $R_1$. Since $\widehat f_1$ is a M\"obius transformation, we conclude
that the Julia set of the semigroup $J_{\widehat H}$ must consist of a single point, giving a
contradiction. Therefore
$$
\widehat f_1(F_{\widehat H})\subsetneq F_{\widehat H}.
$$

For every $k = 2,\dots,d$ the map $\widehat f_k$ has two critical points $\{-b, -1/b\} \subset
F_{\widehat H}$. It follows from the Schwarz-Pick Lemma that for every $k = 1,\dots,d$ the map $\widehat f_k$ is
a contraction with respect to the metric $\rho$.
This means that for every compact set $Q \subset F_{\widehat H}$ we may find a constant $c_{Q} < 1$ so that for
any $z,w \in Q$ we have
$$
\rho\left(\widehat f_k(w),\widehat f_k(z)\right) \le c_{Q} \cdot \rho\left(w,z\right), \qquad k=1,\dots,d.
$$
Let $K$ as in \eqref{setK}, $r \in \mathbb{R}_{> 0}$ and define $Q = \overline{B}_{\rho}(K,r)$, where the
ball is taken with respect to the hyperbolic metric of $F_{\widehat H}$. Given $z\in K$ and $w \in B_\rho(z,r)$,
we know that for every $k,n \ge 0$ we have $\widehat f^n_{\omega_k}(z)\in K$ and
$$
\rho\left(\widehat f^{n}_{\omega_k}(w),\widehat f^{n}_{\omega_k}(z)\right) \le \rho(w,z).
$$
This shows that $\widehat f^{n}_{\omega_k}(w),\widehat f^{n}_{\omega_k}(z)\in Q$, and thus that
$$
\rho\left(\widehat f^{n_{k}}_{\omega_k}(w),\widehat f^{n_{k}}_{\omega_k}(z)\right)\le c_{Q}^{n_k} \cdot
\rho\left(w,z\right),
$$
which finally implies that $g_\infty(w) = g_\infty(z)$ for every $w \in B_\rho(z,r)$. We took $r$ to be arbitrary
and thus we can conclude that $g_\infty$ is constant. For every $g \in \widehat H$ and $z \in K$ it follows that
$g(z) \in K$, completing the proof.
\end{proof}

We recall the following definition of hyperbolicity for semigroups, introduced in \cite{S1,S2}.
\begin{definition}
Let $G$ be a rational semigroup and consider the \textit{postcritical set}
$$
P_G:=\overline{\bigcup_{g \in G}\{\text{ critical values of }g\}}.
$$
We say that the semigroup in \textit{hyperbolic} if $P_G\subset F_G$.
\end{definition}

\begin{proposition}
The semigroup $\widehat{H}$ is hyperbolic.
\end{proposition}
\begin{proof}
In our case the postcritical set can be written as
$$
P_{\widehat H} = \overline{\bigcup_{\substack{\omega \in \Omega \\ n\in\mathbb N}}\widehat
	f^n_\omega(\{- b ,-1/ b \})}.
$$
It is clear that $f^n_\omega(\{-b,-1/b\})\in \widehat{\mathbb C}\setminus \mathbb S^1$ for
every $n$ and $\omega$. Furthermore, by the previous theorem we know that every limit point belongs
to $K$, showing that
$$
P_{\widehat H} \subset (\widehat{\mathbb C} \setminus \mathbb S^1) \cup K \subset F_{\widehat H}
$$
\end{proof}

Given $\omega \in \Omega$ we write $F_\omega, J_\omega$ for the Fatou and Julia sets of the family
$\{\widehat f^n_\omega\}$. By \eqref{Juliasetinclusion} we have that
$$
J_\omega\subset J_{\widehat H} \subset \mathbb S^1 \setminus I
$$
and one can show that if $z \in J_\omega$ its orbit avoids the set $I$. On the other hand given
$z \in F_\omega \cap \mathbb S^1$, by Proposition \ref{convergencetoK} we can find $n > 0$ so that
$\widehat f^n_\omega(z) \in I$. We conclude that we can write
\begin{equation}
\label{Juliasetsequence}
J_\omega = \bigcap_{n\in\mathbb N} (\widehat f^n_\omega)^{-1}(\mathbb S^1\setminus I).
\end{equation}

The following lemma shows expansivity of the dynamics on the Julia set $J_\omega$ of every sequence $\omega$.
The result corresponds to \cite[Theorem 2.6]{S2}; for the sake of completeness we provide a sketch the proof.
\begin{lemma}
\label{firstlemmaexpansion}
Let $\lambda \in \Arc[\lambda_0,\lambda_2)$ and $\kappa>1$. Then there exists a positive integer
$N \ge 1$ so that for any $\omega\in \Omega$ and $z \in J_\omega$  we have
$$
|(\widehat f^{N}_{\omega})'{(z)}| \ge \kappa.
$$
\end{lemma}
\begin{proof}[Sketch of the proof.] We start by choosing an open simply connected neighborhood
$V \subset (\widehat{\mathbb C} \setminus \mathbb S^1) \cup I$ containing $K$ and two open simply connected
neighborhoods $U' \Subset U $ containing $\mathbb S^1 \setminus I$ and disjoint from $P_{\widehat H}$ and $V$.
Then choose $C > 1$ so that $d\rho_{U'} \ge C \cdot d\rho_U$, where we write $d\rho_U, d\rho_{U'}$ for the
infinitesimal hyperbolic metric of the two sets.

By Proposition \ref{convergencetoK} there exists a positive integer $N_0$ so that
$$\widehat f^{N_0}_\omega (\overline U\setminus U')\subset V,\qquad \forall\,\omega\in\Omega,\, \forall n\ge N_0.$$

Given $\omega\in \Omega$ and $z\in J_\omega$, by \eqref{Juliasetsequence}, we have $\widehat f^n_\omega(z)\in U'$
for every $n$.
Therefore we may define $U_0\subset U'$ as the connected component of $(\widehat f^{N_0}_\omega)^{-1}(U)$ containing $z$.
Since $U$ is simply connected and disjoint from the postcritical set, the map $\widehat f^{N_0}_\omega$ preserves the
hyperbolic metrics of $U_0$ and $U$. Since $U_0 \subset U'$, this implies that
$$
|(\widehat f^{N_0}_\omega)'(z)|_U \ge C.
$$
The hyperbolic metric of $U$ and the Euclidean metric are comparable on $U'$. Therefore by taking an integer $k$ sufficiently
large, which does not depend on the choice of the sequence $\omega$ and $z$, we conclude that the value $N=kN_0$ satisfies the
requirements of the Lemma.
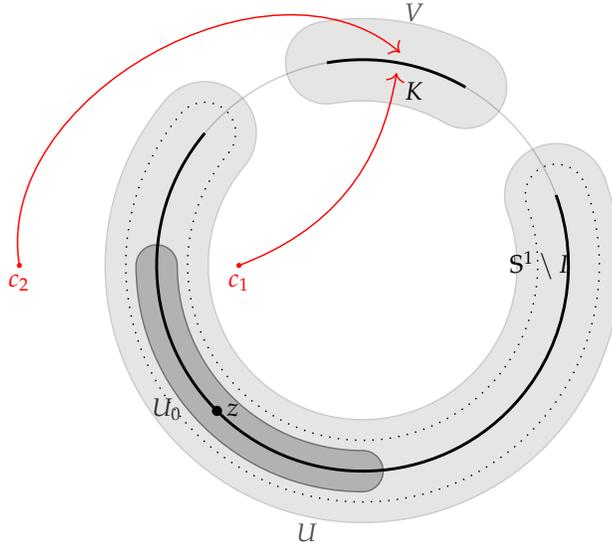
\begin {figure}[h!]
	\centering
	\resizebox {.7\columnwidth} {!} {
		\begin {tikzpicture}[scale = 2]
            \path[draw=black,fill=gray,opacity=.2] ([shift=(-40:.25)] -0.766,0.643) arc (-40:140:.25) -- ([shift=(140:.25)] -0.766,0.643) arc (140:380:1.25) -- ([shift=(380:.25)] .94,0.342) arc (380:560:.25) --([shift=(560:.25)] .94,0.342) arc (20:-220:.75) --cycle;

			\path[draw=black,dotted] ([shift=(-40:.15)] -0.766,0.643) arc (-40:140:.15) -- ([shift=(140:.15)] -0.766,0.643) arc (140:380:1.15) -- ([shift=(380:.15)] .94,0.342) arc (380:560:.15) --([shift=(560:.15)] .94,0.342) arc (20:-220:.85) --cycle;

			\path[draw=black,fill=gray,opacity=.2] ([shift=(100:.2)] -0.174,0.985) arc (100:280:.2) -- ([shift=(280:.2)] -.174,.985) arc (100:60:.8) -- ([shift=(-120:.2)] .5,.866) arc (-120:60:.2) -- ([shift=(60:.2)] .5,.866) arc (60:100:1.2) --cycle;

			\path[draw=black,fill=gray,opacity=.5] ([shift=(0:.1)] -1,0) arc (0:180:.1) -- ([shift=(180:.1)] -1,0) arc (180:270:1.1) -- ([shift=(270:.1)] 0,-1) arc (270:450:.1) -- ([shift=(90:.1)] 0,-1) arc (270:180:.9) --cycle;

			\draw[thick] ([shift=(140:1)] 0,0) arc (140:380:1);
			\draw[thick] ([shift=(60:1)] 0,0) arc (60:100:1);
			\draw[opacity=.2] ([shift=(100:1)] 0,0) arc (100:140:1);
			\draw[opacity=.2] ([shift=(20:1)] 0,0) arc (20:60:1);
			\begin{scriptsize}
			\coordinate [label=-90:$\color{red}c_1$]  (a) at (-3/5,0) ;
			\coordinate  (a1) at (0,0) ;
			\node at (-3/5,0) [circle,fill,inner sep=.5pt,red]{};
			\coordinate [label=-90:$\color{red}c_2$]  (b) at (-5/3,0) ;
			\coordinate  (b1) at (-1.5,0) ;
			\node at (-5/3,0) [circle,fill,inner sep=.5pt,red]{};
			\coordinate  (c) at (.95*.174,.95*.985) ;
			\coordinate  (d) at (1.05*.174,1.05*.985) ;
			\draw[->,red] (a) to [bend right] (c);
			\draw[->,red] (b) to [bend left=70] (d);
			\node[opacity=.7] at (.25,1.23) {$V$};
			\node[opacity=.7] at (-0.27,-1.3) {$U$};
			\node[opacity=.8] at (-.95,-.7) {$U_0$};
			\node at (.25,.85) {$K$};
			\node at (.86,0) {$\mathbb S^1\setminus I$};
			\node at (-.707,-.707) [circle,fill,inner sep=1pt]{};
			\node at (-.63,-.7) {$z$};
			\end{scriptsize}
		\end {tikzpicture}
    }
    \caption{ Sets used in the proof of Lemma \ref{firstlemmaexpansion}}
\end {figure}
\end{proof}

We will now show that, once we are bounded away from $\lambda_2$, the value of $N$ can be chosen independently from $\lambda$.
It will be convenient to reintroduce the subscript $\lambda$ in order to distinguish between different parameter values.
It is clear the set $J_\omega$ is dependent on $\lambda$, thus it will be denoted as $J_{\omega,\lambda}$.
\begin{lemma}
Let $\lambda' \in \Arc[\lambda_0,\lambda_2)$ and $\kappa>1$. Then there exists a positive integer $N \ge 1$ so that
for any $\lambda \in \Arc[\lambda_0,\lambda']$, any $\omega \in \Omega$ and any $z \in J_\omega$ there exists $1 \leq n \le N$ so
that
$$
|(\widehat f^n_{\omega,\lambda})'(z)|\ge \kappa.
$$
\end{lemma}

\begin{proof}
Given $\lambda\in \Arc[\lambda_0,\lambda_2)$ let $N_\lambda$ be the minimum integer for which the previous lemma is valid.
Suppose now that there exists a sequence $\lambda_k \in \Arc[\lambda_0,\lambda']$ such that $N_k = N_{\lambda_k} \to \infty$.
It follows that we may find sequences $\omega_k \in \Omega$ and $z_k \in J_{\omega_k,\lambda_k}$, so that
$$
|(\widehat f^n_{\omega_k,\lambda_k})'(z_k)|<\kappa, \qquad\forall \,0\le n\le N_k-1.
$$
By passing to a subsequence if necessary, we may assume that the three following conditions are satisfied
\begin{enumerate}
\item the parameters $\lambda_k$ converge to $\lambda_\infty\in  \Arc[\lambda_0,\lambda_1]$;
\item the points $z_k$ converge to $z_\infty\in \mathbb S^1$;
\item the sequences $\omega_k$ and $\omega_{k+1}$ agree on the first $k$ elements.
\end{enumerate}

Since the arc $I_\lambda$ varies continuously in $\lambda$ and by \eqref{Juliasetsequence} every $z_k \in \mathbb S^1
\setminus I_{\lambda_k}$ it follows that $z_\infty \in \mathbb S^1\setminus I_{\lambda_\infty}$. Let $\omega_\infty$ be
the sequence given by $\omega_{\infty,k} = \omega_{k,k}$, where $\omega_{k,n}$ denotes the $n$-th element of the sequence
$\omega_k$. Then it is clear that $\omega_k$ and $\omega_{\infty}$ agree on the first $k$ elements. Given $n\in\mathbb N$ we have
$$
\widehat f^n_{\omega_\infty,\lambda_\infty}(z_\infty)=\lim_{k\to\infty}\widehat f^n_{\omega_k,\lambda_k}(z_k)\in
	\mathbb S^1\setminus I_{\lambda_\infty},
$$
proving that $z_\infty \in J_{\omega_\infty,\lambda_\infty}$. Furthermore, when $k$ is sufficiently large we have
$n\le N_k$, therefore
$$
|(\widehat f^n_{\omega_\infty,\lambda_\infty})'(z_\infty)|=\lim_{k\to\infty}|(\widehat f^n_{\omega_k,\lambda_k})'(z_k)|\le \kappa,
$$
contradicting the previous lemma.
\end{proof}

\begin{proposition}
\label{expansion}
Let $\lambda'\in \Arc[\lambda_0,\lambda_2)$ and $\kappa > 1$. Then there exists a positive integer $N\ge 0$
so that for any $\lambda \in \Arc[\lambda_0,\lambda']$, any $\omega\in\Omega$ and any
$z\in J_{\omega,\lambda}$ we have
$$
|(\widehat f^N_{\omega,\lambda})'(z)| \ge \kappa.
$$
\end{proposition}
\begin{proof}
Let $N$ as in the previous lemma. For any $\lambda \in\mathbb S^1$, any $\omega\in\Omega$ and any $n\in\mathbb N$
the rational map $\widehat f^n_{\omega,\lambda}$ has no critical points on the unit circle. We may therefore find a
constant $\varepsilon > 0$ such that
$$
|(\widehat f^n_{\omega,\lambda})'(z)| > \varepsilon,\qquad\forall \lambda,z\in\mathbb S^1,\quad\forall\,
	\omega\in\Omega,\quad\forall\, n\le N.
$$
Let $j \in \mathbb N$ so that $\varepsilon \cdot \kappa^j > \kappa$.

Suppose now that $\lambda\in \Arc[\lambda_0,\lambda']$, that $\omega\in\Omega$ and that
$z\in J_{\omega,\lambda}$. Thanks to the previous lemma there exist positive integers $J \ge j$ and
$n_1, \dots, n_J \in \{1,\dots, N\}$ which satisfy $(j-1) \cdot N < n_1+\dots +n_J\le jN$ and so that the following
holds: if we write $m_0 = 0$ and $m_i = n_1+\dots n_i$ then for $i \in \{1, \dots, J\}$
$$
|(\widehat f^{m_{i},m_{i-1}}_{\omega,\lambda})'(\widehat f^{m_{i-1}}_{\omega,\lambda}(z))|\ge \kappa,
$$
showing that
$$
|(\widehat f^{jN}_{\omega,\lambda})'(z)|\ge
	|(\widehat f^{jN,jN - m_j}_{\omega,\lambda})'(\widehat f^{m_j}_{\omega,\lambda}(z))| \cdot \kappa^J
	\ge \varepsilon \cdot \kappa^j \ge \kappa.
$$
By choosing $jN$ instead of $N$, we conclude the proof of the proposition.
\end{proof}

\subsection{Existence of expanding sequences}
\begin{lemma}
\label{lem: modulolemma}
Given $d \in \mathbb{Z}_{\geq 2}$, $m \in \left\{1, \dots, d - 1\right\}$,
$t \in (0,1)$ and $s \in (0, t)$ there exists a $k \in \left\{1, \dots, d-1
\right\}$ such that
\[
	A_{k} \coloneqq \left(\frac{2m - s}{d}\right) \cdot k + t
\]
is an element of the interval $(1,2)$ when reduced modulo $2$.
\end{lemma}
\begin{proof}
Note that either $m < (d+s)/2$ or $m > (d+s)/2$, since $s \notin \mathbb Z$.
We consider these two cases separately.

If $m < (d+s)/2$ then $A_{k+1} - A_{k} < 1$ for
all $k$. It follows that for any open interval $(a,a+1)$ contained in
the interval $[A_{0},A_{d}]$, there exists an integer $0 < k < d$ such
that $A_{k} \in (a,a+1)$. Since $A_{0} = t < 1$ and
$A_{d} = 2m + t -s > 2$, there exists a $k$ such that
$A_{k} \in (1,2)$.
		
Now assume that $m > (d+s)/2$ and define
\[
\tilde{A}_{p} = \left(\frac{2(d-m) + s}{d}\right) \cdot p + t - s.
\]
Observe that $A_{k}-\tilde{A}_{d-k} = 2(k + m - d)$,
and thus $A_{k}~\equiv~\tilde{A}_{d-k} \pmod{2}$. Therefore it suffices to
find a $p \in \{1, \dots, d-1\}$ for which $\tilde{A}_p \in (1,2)$. Such
a $p$ can be found by the same argument as above, because $\tilde{A}_0
= t-s < 1$, $\tilde{A}_d = 2 (d - m) + t > 2$ and $\tilde{A}_{p+1} -
\tilde{A}_p < 1$.
\end{proof}

\begin{proposition}
\label{prop: Omega_z not empty}
Let $d \ge 2$ and $\lambda \in  \Arc[\lambda_0,\lambda_1)$. Then for every
$z\in \mathbb S^1 \setminus I$ there exists $\omega\in\Omega$ so that $z\in J_{\omega}$.
\end{proposition}

\begin{proof}
Choose $0 < t < 1$ so that $\lambda = e^{i \pi t}$. The function $f_d$ can be
written as $f_d(z) = \lambda \cdot \gamma(z)^d$, where $\gamma(z) = (z + b)/(b z + 1)$
is a M\"obius transformation that fixes the unit circle. It follows that we can find
$d$ disjoint sets $J_0, \dots, J_{d-1}$ such that
\[
f_d^{-1}(\Arc(1,\lambda)) = J_0 \cup \cdots \cup J_{d-1},
\]
ordered in such a way that for all $m \in \{0, \dots, d-1\}$ we have
\[
\gamma(J_m) = \Arc{\left(\exp{\left(i \pi \frac{2m - t}{d}\right)},
						\exp{\left(i \pi \frac{2m}{d}\right)}
				 \right)}.
\]
Since $f_d$ inverts the orientation of the unit circle and $f_d(1)=\lambda$,
we have for $z\in \Arc(1,\lambda)$ close enough to $1$ that $f_d(z)\in \Arc(1,\lambda)$.
Furthermore by Theorem \ref{l0l1theorem} we cannot have $f_d(\lambda)\in \Arc(1,\lambda)$.
This shows that one of the connected component of $f_d^{-1}{\left(\Arc(1,\lambda)\right)}$
is of the form $J = \Arc(1,z') \subset \Arc(1,\lambda)$. This component must contain
the arc $I$. Since $\gamma(1) = 1$, we find that $J = J_0$.

Now let $J_0', \dots, J_{d-1}'$ denote the inverse arcs of $I$ under
the map $f_d$ in such a way that $J_m' \subset J_m$ for all $m$. Then we see
that $J_0' = I$. We now present a way to choose, given a $z \in \mathbb{S} \setminus I$,
an integer $k$ such that $\widehat{f_k}(z) \in \mathbb{S} \setminus I$.
\begin{itemize}
	\item
	If $z \not \in f_d^{-1}(I)$, we can choose $k = d$ because then $\widehat{f_{k}}(z)
	=f_d(z) \not \in I$.
	\item
	If $z \in f_d^{-1}(I) = J_0' \cup \cdots \cup J_{d-1}'$ but $z \not \in I = J_0$,
	we see that $z \in J_m$ for some $m \in \{1, \dots, d-1\}$ and thus we can write
	\[
		\gamma(z) = \exp{\left(i \pi \frac{2m - s}{d}\right)}
	\]
	for some $0< s < t$. We find that
	\[
		\widehat{f_k}(z) = f_1\left(\lambda \cdot \gamma(z)^k\right)
			= f_1{\left(\exp{\left(i \pi \left(\frac{2m - s}{d}\cdot k + t\right)\right)}\right)}.
	\]
	According to Lemma \ref{lem: modulolemma}, we can choose $k$ in such a way that
	$\lambda \cdot \gamma(z)^k \in \Arc(-1,1)$. Since $f_1(\Arc(-1,1)) = \Arc(\lambda, -\lambda)$,
	we conclude that $\widehat{f_k}(z) \not \in I$.
\end{itemize}
This procedure defines the first $n$ steps of the sequence $\omega$ and satisfies
$\widehat f^k_\omega(z) \in \mathbb S^1\setminus I$ for all $k\le n$. By iterating this procedure for
the point $\widehat f^n_\omega(z)$ we find a sequence $\omega \in \Omega$ such that $\widehat f^n_\omega(z)\in\mathbb S^1\setminus I$ for all $n\in\mathbb N$. By \eqref{Juliasetsequence} we conclude that $z\in J_\omega$.
\end{proof}

Let $\lambda \in  \Arc[\lambda_0,\lambda_1)$ and $z \in \Arc[\lambda,1]$. Then we can modify
the procedure described in the proof above to obtain  the following:
\begin{proposition}
\label{nonemptyomegap}
Let $d \ge 2$ and $\lambda \in \Arc[\lambda_0,\lambda_1)$. Then for every $z \in \Arc[\lambda,1]$
there exists $\omega \in \Omega$ so that $z\in J_\omega$ and
$$
\widehat f^n_\omega(z)\in \Arc[\lambda,1],\qquad\forall n\in\mathbb N
$$
\end{proposition}

\begin{corollary}
\label{Juliasetsemigroup}
For every $\lambda\in  \Arc[\lambda_0,\lambda_2)$ we have
$$
\widehat{\mathbb C}\setminus(\mathbb S^1\setminus I) =F_{\widehat H},\qquad \mathbb S^1 \setminus I
												= J_{\widehat H}.
$$
\end{corollary}


\begin{proof}
Given $z\in \mathbb S^1 \setminus I$ choose $\omega\in \Omega$ so that $z\in J_\omega$,
which exists according to Proposition~\ref{prop: Omega_z not empty}.
For every positive integer $k \in \mathbb N$ by Proposition \ref{expansion} we may find some $N_k$ such
that
$$
|(\widehat f^{N_k}_\omega)'(z)|\ge k,
$$
showing that $|(\widehat f^{N_k}_\omega)'(z)| \to \infty$ and therefore that
$z \in J_{\widehat H}$. The two equalities follow from \eqref{Juliasetinclusion}.
\end{proof}

\section{Zeros for the semigroup}

Our goal in this section is to give a more precise description of the zeros of $Z_G(\lambda)$ in $\Arc[\lambda_0, \lambda_1]$ for trees $G$ that are spherically symmetric of bounded degree. We will prove case (3) of Theorem \ref{thm:main}:

\begin{theorem}\label{zeroparamonSsemi}
Let $d\in\mathbb N_{\ge 2}$. Then there exists $\lambda_3\in \Arc(\lambda_0,-1]$ so that the set of zero parameters for spherically symmetric trees of degree $d$ contains a dense subset of $ \Arc[\lambda_0,\lambda_3]$.
\end{theorem}

We will also prove the following weaker statement, which we will however prove for a considerably larger circular arc.

\begin{theorem}\label{zeroparamsemi}
Let $d\in\mathbb N_{\ge 2}$. Then the closure of the set of zero parameters of spherically symmetric trees of degree $d$ contains $ \Arc[\lambda_0,\lambda_2]$.
\end{theorem}


\subsection{Proof of Theorem \ref{zeroparamonSsemi}}
Since we will repeatedly deal with distinct values of the parameter $\lambda$, we will always use the subscript $\lambda$ in order to specify the map $f_\lambda$ that we are using.

Choose a parameter value $\lambda'\in \Arc(\lambda_0,\lambda_2)$. By Proposition \ref{expansion} there exists a positive integer $N$ so that for every $\lambda\in  \Arc[\lambda_0,\lambda']$, any $\omega\in\Omega$ and any $z\in J_{\omega,\lambda}$ we have
\begin{equation}
\label{secondexpansion}
|(\widehat f ^N_{\omega,\lambda})'(z)|> 3.
\end{equation}
Having fixed $N$, it follows from the compactness of $\mathbb S^1$ that there exists a constant $C>0$ so that for any $z\in\mathbb S^1$ and any $\omega\in\Omega$ we have
$$
\left|\widehat f^N_{\omega,\lambda}(z)-\widehat f^N_{\omega,\mu}(z)\right|\le C|\lambda-\mu|\qquad\forall\,\lambda,\mu\in\mathbb S^1.
$$

\begin{lemma}
There exist $\lambda_3\in  \Arc[\lambda_0,\lambda']$  and a positive integer $M$ so that for every $\lambda\in \Arc[\lambda_0,\lambda_3]$ and $1\le m\le M$ we have
\begin{equation}
\label{initialexpansion}
f^{m}_{d,\lambda}(1)\in  \Arc[\lambda,1]\quad\mathrm{and}\quad|\partial_\lambda \,f^{M}_{d,\lambda}(1)|> 2C+1.
\end{equation}
\end{lemma}
\begin{proof}
The point $1$ is a repelling periodic point of order $2$ for the map $f_{d,\lambda_0}$, meaning that $|(f^2_{d,\lambda_0})'(1)|>1$. Writing $\gamma(z)=\frac{z+ b}{ b  z+1}$ we obtain
\begin{align*}
\left(\partial_\lambda\, f^{2m}_{d,\lambda}(1)\right)\big|_{\lambda=\lambda_0}&=\left(\partial_\lambda\, f^2_{d,\lambda}(1)\right)\big|_{\lambda=\lambda_0}\left(1+(f^2_{d,\lambda_0})'(1)+\dots+(f^{2}_{d,\lambda_0})'(1)^{m-1}\right)\\
&=\left(\gamma(\lambda_0)^d+f'_{d,\lambda_0}(1)\right)\frac{(f^{2}_{d,\lambda_0})'(1)^{m}-1}{(f^{2}_{d,\lambda_0})'(1)-1}.
\end{align*}
Observing that $\gamma(\lambda_0)^d+f'_{d,\lambda_0}(1) \neq 0$ it follows that when $M=2m$ for $m$ sufficiently large:
$$
\left|\left(\partial_\lambda\, f^M_{d,\lambda}(1)\right)\big|_{\lambda=\lambda_0}\right|>2C+1.
$$
By continuity there exists $\lambda_3\in  \Arc[\lambda_0,\lambda']$ so that
$$
\left|\partial_\lambda\, f^M_{d,\lambda}(1)\right|>2C+1,\qquad\forall \,\lambda\in  \Arc[\lambda_0,\lambda_3].
$$

Given $\lambda\in \Arc(\overline{\lambda_1},\lambda_1)$ we write $z_\lambda,w_\lambda$ for the two boundary points of $I_\lambda$. These boundary points form a repelling $2$-cycle, hence by the implicit function theorem they vary holomorphically in an open neighborhood of $\Arc(\overline{\lambda_1},\lambda_1)$. By exchanging the order of $z_\lambda$ and $w_\lambda$ if necessary, we may assume that $z_{\lambda_0}=1$. By Theorem \ref{l0l1theorem} for every $\lambda\in \Arc(\lambda_0,\lambda_1)$ we have $z_\lambda\in \Arc (1,\lambda)$.

The point $f_{d,\lambda}(1)=\lambda$ clearly belongs to $ \Arc[\lambda,1]$. Recall that the map $f_{d,\lambda}^2$ preserves the orientation of the unit circle. Since $z_\lambda$ is a repelling fixed point, by replacing $\lambda_3$ with a parameter in $ \Arc(\lambda_0,\lambda_2)$ sufficiently close to $\lambda_0$ so that the point $z_\lambda$ remains close to $1$ for every $\lambda\in  \Arc[\lambda_0,\lambda_3]$,  we may assume that $f^2_{d,\lambda}(1)\in  \Arc[\lambda,1]$. Up to replacing at each step $\lambda_3$ with a parameter closer to $\lambda_0$, we may therefore assume that also $f^3_{d,\lambda}(1),\dots, f^M_{d,\lambda}(1)\in  \Arc[\lambda,1]$, concluding the proof of the lemma.
\end{proof}

By Proposition \ref{badparrepelling} we know that zero parameters $\lambda \in \overline{\mathcal{Z}_{\mathcal{C}_{d+1}} \cap \mathbb S^1}$ accumulate on $\lambda_0$. From now until the end of the proof of Theorem \ref{zeroparamonSsemi} we will fix the value of the parameter $\lambda\in \Arc(\lambda_0,\lambda_3]$.

By Theorem \ref{l0l1theorem} we know that $I_\lambda\Subset  \Arc(1,\lambda)$. This fact together with \eqref{initialexpansion} and \eqref{secondexpansion} implies the existence of the following constant:

\begin{definition}[Choice of the constant $\varepsilon>0$] There exists a sufficiently small constant $\varepsilon>0$ so that the following three conditions are satisfied:
\begin{itemize}
	\item[a.] The distance between the attracting interval $I_\lambda$ and $ \Arc[\lambda,1]$ is at least $\varepsilon$, $$\inf_{\substack{z\in I_\lambda\\w\in  \Arc[\lambda,1]}}|z-w|>\varepsilon.$$
	\item[b.] For any $\mu\in B(\lambda,\varepsilon)$ we have
	$$
	|f^M_{d,\lambda}(1)- f_{d,\mu}^M(1)|>2C|\lambda-\mu|;
	$$
	\item[c.] For any $\omega\in \Omega$, any $z\in J_{\omega,\lambda}$ and any $w\in B(z,\varepsilon)$ we have
	$$
	|\widehat f^N_{\omega,\lambda}(z)-\widehat f^N_{\omega,\lambda}(w)|>2|z-w|.
	$$
\end{itemize}
\end{definition}

\begin{definition}[Choice of the sequence $\sigma\in\Omega$] we define a sequence of the form
$$
\sigma=(\underbrace{d,\dots,d}_{M},\underbrace{\sigma_{M+1},\sigma_{M+2},\dots}_{\sigma^{0}}),
$$
where $\sigma^0\in \Omega$ is chosen such that $\widehat f^n_{\sigma,\lambda}(1)\in \Arc[\lambda,1]$ for all $n\in\mathbb N$. The existence of such sequence is guaranteed by Proposition \ref{nonemptyomegap}, and we have $1\in J_{\sigma,\lambda}$.
\end{definition}

\begin{lemma}
\label{potafiga}
For every $\mu\in\mathbb S^1\setminus\{\lambda\}$ there exists a positive integer $n_\mu$ so that
$$
|\widehat f^{n_\mu}_{\sigma,\lambda}(1)-\widehat f^{n_\mu}_{\sigma,\mu}(1)|\ge\varepsilon
$$
\end{lemma}
\begin{proof}
Suppose on the other hand that there exists $\mu\in\mathbb S^1\setminus\{\lambda\}$ so that $|\widehat f^n_{\omega,\lambda}(1)-\widehat f^n_{\omega,\mu}(1)|<\varepsilon$ for every $n\in\mathbb N$.

Since $\widehat f_{\sigma,\mu}(1)=\mu$, it follows in particular $|\lambda-\mu|<\varepsilon$. Therefore after the first $M$ steps we obtain that
$$
|\widehat f^M_{\sigma,\lambda}(1)-\widehat f^M_{\sigma,\mu}(1)|>2C|\lambda-\mu|.
$$

Write $T:\Omega\rightarrow \Omega$ for the left shift map, and define $\sigma^j=T^{M+jN}(\sigma)$ and $z_{j,\mu}=\widehat f^{M+jN}_{\sigma,\mu}(1)\in\mathbb S^1$. Thanks to the choice of the sequence $\sigma^0$ it follows that $z_{j,\lambda}\in  \Arc[\lambda,1]$ and that
$$
z_{j,\lambda}\in J_{\sigma^j,\lambda}.
$$

We claim that for every $j\in\mathbb N$ we have $|z_{j,\lambda}-z_{j,\mu}|>\alpha_n C |\lambda-\mu|$, where $\alpha_0=2$ and $\alpha_{n+1}=2\alpha_n-1$. The claim certainly holds for $j=0$. We will now assume it holds for $j\in\mathbb N$ and prove that it holds for $j+1$.

By the assumption on $\mu$ we know that $|z_{j,\lambda}-z_{j,\mu}|<\varepsilon$ for every $j\in\mathbb N$. Therefore, thanks to the choice of the constants $\varepsilon$ and $C$, we conclude that
\begin{align*}
|z_{j+1,\lambda}-z_{j+1,\mu}|&=|\widehat f^{M+(j+1)N}_{\sigma,\lambda}(1)-\widehat f^{M+(j+1)N}_{\sigma,\mu}(1)|\\&>|\widehat f^{N}_{\sigma_j,\lambda}(z_{j,\lambda})-\widehat f^{N}_{\sigma^j,\lambda}(z_{j,\mu})|-|\widehat f^{N}_{\sigma^j,\lambda}(z_{j,\mu})-\widehat f^{N}_{\sigma^j,\mu}(z_{j,\mu})|\\
&>2|z_{j,\lambda}-z_{j,\mu}|-C|\lambda-\mu|\\
&>(2\alpha_n-1)C|\lambda-\mu|.
\end{align*}

It then follows that for every $n$ we must have $\alpha_n C|\lambda-\mu|<\varepsilon$, which is possible only if $\lambda=\mu$, giving a contradiction.
\end{proof}

\begin{lemma}
\label{aluraenculet}
There exists a positive integer $n_0$ so that every point in $\mathbb S^1\setminus I_\lambda$ lies at distance at most $\varepsilon/4$ from the set
$$
\Sigma:=\bigcup_{\substack{\omega\in\Omega\\ 0\le n\le n_0}}\widehat f^{-n}_{\omega,\lambda}\left(\{-1\}\right).
$$
\end{lemma}
\begin{proof}
By Corollary \ref{Juliasetsemigroup} we have $\mathbb S^1\setminus I_\lambda\subset J_{\widehat{H}_\lambda}$. The semigroup $\hat{H}_\lambda$ has no exceptional points, therefore the collection of all preimages of a given point in $\hat{\mathbb C}$ accumulates on the whole Julia set, and therefore on $\mathbb S^1\setminus I_\lambda$.

Writing $V=\bigcup_{g\in\widehat{H}_\lambda}g^{-1}\left(\{-1\}\right)$ we therefore conclude that
$$
\mathbb S^1\setminus I_\lambda\subset\overline{V}\subset\bigcup _{z\in K}B(z,\varepsilon/4).
$$
By compactness of $\mathbb S^1\setminus I_\lambda$ we may find $w_1,\dots, w_\nu\in V$ so that $B(w_1,\varepsilon/4)\cup\dots\cup B(w_n,\varepsilon/4)$ still covers the set $\mathbb S^1\setminus I_\lambda$. Every $w_j$ is the preimage of some element $g_j\in\widehat{H}_\lambda$ of length $n_j<\infty$, meaning that there exists $\omega_j\in\Omega$ so that
$$
\widehat f^{n_j}_{\omega_j,\lambda}(z_j)=-1.
$$
By taking $n_0=\max{n_i}$ we obtain
$$
\{z_1,\dots,z_n\}\subset \bigcup_{\substack{\omega\in\Omega\\ 0\le n\le n_0}}\widehat f^{-n}_{\omega,\lambda}\left(\{-1\}\right),
$$
concluding the proof of the lemma.
\end{proof}

Write $\Sigma=\{w_1,\dots, w_{\nu}\}$. For every $w_j\in\Sigma$ there exists $\omega_j\in\Omega$ and $0\le n_j\le n_0$ so that
$$
\widehat f^{\nu_j}_{\omega_j,\lambda}(w_j)=-1.
$$
No element in the semigroup $\widehat{H}_\lambda$ can have critical points on the unit circle. Therefore by the implicit function Theorem there exists a holomorphic map $\mu\mapsto w_{j,\mu}$ defined in a neighborhood of $\lambda$ so that
$$
\widehat f^n_{\omega,\mu}(w_{j,\mu})=-1,\qquad w_{j,\lambda}=w_j.
$$
By taking $\delta_0>0$ sufficiently small we may assume that for every $j=1,\dots,\nu$ the map $w_{j,\mu}$ is defined on $B(\lambda,\delta_0)$ and that
$$
|w_{j,\mu}-w_{j,\lambda}|<\varepsilon/4,\qquad \forall\, \mu\in B(\lambda,\delta_0),\,\,\forall\, j=1,\dots,\nu.
$$

Given $\delta<\delta_0$ we choose $\mu'\in \mathbb S^1\cap B(\lambda,\delta)^*$ in such  a way that $ \Arc[\lambda,\mu']\subset B(\lambda,\delta)$. We note that this last condition is not necessary, but simplifies the proof.

By Lemma \ref{potafiga} we may choose $n'=n_{\mu'}$ so that $|\widehat f^{n'}_{\sigma,\lambda}(1)-\widehat f^{n'}_{\sigma,\mu}(1) |\ge \varepsilon$. Given $\mu\in  \Arc[\lambda,\mu']$, write $z_\mu=\widehat f^{n'}_{\sigma,\mu}(1)$.

\begin{proposition}
There exists $j\in\{1,\dots,\nu\}$ and $\mu\in \Arc[\lambda,\mu']$ so that $z_\mu=w_{j,\mu}$
\end{proposition}
\begin{proof}
By the definition of $\sigma$ we have that $z_{\lambda}\in \Arc[\lambda,1]$ and by the choice of the constant $\varepsilon$ we know that $d(I_\lambda,z_\lambda)>\varepsilon$. Therefore, by replacing $\mu'$ with another parameter closer to $\lambda$ in such a way that the value of $n'$ does not change, we may further assume that $z_\mu\in \mathbb S^1\setminus I_\lambda$ for every $\mu\in  \Arc[\lambda,\mu']$.

The image of the map $\Arc[\lambda,\mu']\ni \mu\mapsto z_\mu$ contains either $ \Arc[z_\lambda,z_{\mu'}]$ or $ \Arc[z_{\mu'},z_\lambda]$ (notice that one of the two possibility occurs, since the image does not intersect $I_\lambda$). We will prove the proposition assuming that the first case occurs, a similar proof works in the other case.

Choose a point $\zeta\in  \Arc[z_\lambda,z_\mu]$ so that the distance of the point from both the extrema of the arc is bigger or equal to $\varepsilon/2$, which is possible since $|z_\lambda-z_\mu|\ge \varepsilon$. Let $j\in\{1,\dots, \nu\}$ so that $|w_{j\lambda}-\zeta|<\varepsilon/4$. It then follows that $w_{j,\lambda}\in \Arc[z_\lambda,z_{\mu'}]$ and that
$$
|w_{j,\lambda}-z_\lambda|>\varepsilon/4,\qquad |w_{j,\lambda}-z_{\mu'}|>\varepsilon/4,
$$
and thanks to the fact that $\delta<\delta_0$ we conclude that
$$
w_{j,\mu}\in  \Arc[z_\lambda,z_{\mu'}],\qquad\forall\, \mu\in \Arc[\lambda,\mu'],
$$
and therefore that there exists $\mu\in \Arc[\lambda,\mu']$ so that $w_{j,\mu}=z_\mu$.
\end{proof}

Let $\mu\in\mathbb S^1\cap B(\lambda,\delta)$ and $j$ as in the previous lemma, and let $\omega_j\in\Omega$ and $n_j\in\mathbb N$ so that $\widehat f^{n_j}_{\omega_j,\mu}(w_{j,\mu})=-1$. We conclude that
\begin{align*}
\widehat f^{n_j}_{\omega_j,\mu}(\widehat f^{n'}_{\sigma,\mu}(1))&=\widehat f^{n_j}_{\omega_j,\mu}(z_\mu)\\
&=\widehat f^{n_j}_{\omega_j,\mu}(w_{j,\mu})\\
&=-1,
\end{align*}
concluding the proof of Theorem \ref{zeroparamonSsemi}.

\subsection{Proof of Theorem \ref{zeroparamsemi}}

Let $\lambda'\in \Arc[\lambda_0,\lambda_2)$. By Proposition \ref{expansion} there exists a positive integer $N\ge 0$ so that for any $\lambda\in \Arc[\lambda_0,\lambda']$ any $\omega\in\Omega$ and any $z\in J_{\omega,\lambda}$ we have
$$
|(\widehat{f}^N_{\omega,\lambda})'(z)|> 3.
$$
Once $N$ is fixed we have the following:

\begin{lemma}
There exist constants $\varepsilon,\delta>0$ so that for any $\lambda\in \Arc[\lambda_0,\lambda']$, any $\omega\in\Omega$, and any $z\in J_{\omega,\lambda}$, there exists a holomorphic map $F_{z,\omega,\lambda}:B(\lambda,\delta)\rightarrow B(z,\varepsilon)$ with $F_{z,\omega,\lambda}(\lambda)=z$, satisfying:
\begin{enumerate}
\item
\begin{equation}
\label{babyK}
|\widehat f^{kN}_{\omega,\lambda}(z)-\widehat f^{kN}_{\omega,\mu}(F_{z,\omega,\lambda}(\mu))|< \varepsilon,\qquad\forall \mu\in B(\lambda,\delta),\,\,\forall k\ge 0,
\end{equation}
\item given any $\mu\in B(\lambda,\delta)$ and any $w\in B(z,\varepsilon)\setminus \{F_{z,\omega,\lambda}(\mu)\}$ there exists a positive integer $k$ so that
$$
|\widehat f^{kN}_{\omega,\mu}(F_{z,\omega,\lambda}(\mu))-\widehat f^{kN}_{\omega,\mu}(w)|\ge 3\varepsilon.
$$
\end{enumerate}
\end{lemma}
\begin{proof}
Note that the second derivative of $\widehat{f}^N_{\omega,\lambda}$ is bounded in a neighborhood of $\mathbb S^1$. It follows that for $\epsilon>0$ and $\delta>0$ sufficiently small, the maps $\widehat{f}^N_{\omega,\lambda}$ are uniformly expanding in a given neighborhood of $J_{\omega, \lambda}$. The existence of the point $F_{z,\omega,\lambda}(\mu)$ follows immediately. The fact that $F_{z,\omega, \lambda}$ can be given as the limit of a sequence of contracting inverse branches implies the holomorphic dependency on $\mu$.
\end{proof}

\begin{proposition}
\label{nonnormal}
Let $\lambda'\in  \Arc[\lambda_0,\lambda_2)$. Then the family of maps
$$\mathcal A:=\{\lambda\mapsto g_\lambda (1)\,|\, g_\lambda\in\widehat{H}_\lambda\}$$
is not normal near $\lambda'$.
\end{proposition}
\begin{proof}
Let $\varepsilon, \delta > 0$ as in the previous lemma. By Lemma \ref{lemmaK} the map $\lambda\mapsto K_\lambda$ is upper semi-continuous in $ \Arc[\lambda_0,\lambda_2)$, and $\lambda\mapsto I_\lambda$ is continuous. Therefore by compactness of $ \Arc[\lambda_0,\lambda']$, and by taking smaller $\varepsilon,\delta$ if necessary, we may assume that
$$
\inf_{\substack{z\in K_{\lambda_2}\\w\in \mathbb S^1\setminus I_{\lambda_1}}}|z-w|>2\varepsilon.
$$
Now assume for the purpose of a contradiction that the family of holomorphic functions $\mathcal A=\{\lambda\mapsto g_\lambda (1)\,|\, g_\lambda\in\widehat{H}_\lambda\}$ is normal near $\lambda'$. By Ascoli--Arzel\'a Theorem it follows that there exists $\delta'<\delta$ so that for every $\mu\in B(\lambda',\delta')$ we have
$$
|\widehat f^n_{\omega,\lambda'}(1)-\widehat f^n_{\omega,\mu}(1)|<\varepsilon\,\qquad\forall \omega\in\Omega,\,\forall n\in\mathbb N.
$$
By Proposition \ref{nonemptyomegap} we can fix  $\omega\in\Omega$ so that $1\in J_{\omega,\lambda}$ and $\widehat f^n_\omega(1)\in \Arc[\lambda,1]$ for every positive integer $n$. It follows that $|\widehat f^{kN}_{\omega,\lambda'}(1)-\widehat f^{kN}_{\omega,\mu}(1)|<\varepsilon$, on the ball $B(\lambda',\delta')$. By the identity principle it follows that $F_{1,\omega,\lambda'}(\mu)=1$ on the bigger ball $B(\lambda',\delta)$, where $F_{1,\omega,\lambda'}$ is the map defined in the previous lemma.

Now notice that given $\lambda''\in B(\lambda',\delta)\cap  \Arc[\lambda_0,\lambda']$ we have $1\in J_{\omega,\lambda''}$. If this were not the case then by \eqref{Juliasetinclusion} we could find a positive integer $n$ so that $\widehat f_{\omega,\lambda''}(1)\in I_{\lambda''}\subset F_{\widehat{\mathcal G}}$, and by Proposition \ref{convergencetoK} we conclude that $\widehat f_{\omega,\lambda''}^n(1)\to K_\lambda''$. In particular when $k$ is sufficiently large the point $\widehat f^{kN}_{\omega,\lambda''}(1)$ lies at distance strictly less than $\varepsilon$ from the set $K_{\lambda''}$. Since instead the point $\widehat f^{kN}_{\omega,\lambda'}(1)$ lies in $\mathbb S^1\setminus I_{\lambda'}$ and the two sets $\mathbb S^1\setminus I_{\lambda'}$ and $K_{\lambda''}$ have distance greater than $2\varepsilon$, we conclude that for $k$ sufficiently large
$$
|\widehat f^{kN}_{\omega,\lambda'}(1)-\widehat f^{kN}_{\omega,\lambda''}(1)|\ge\varepsilon,
$$
contradicting the fact that $F_{1,\omega,\lambda'}(\lambda'')=1$.

On the intersection $U=B(\lambda',\delta)\cap B(\lambda'',\delta)$ the maps $F_1=F_{1,\omega,\lambda'}$ and $F_2=F_{1,\omega,\lambda''}$ are well defined. For every $\mu\in U$ and every positive integer $k$ we have $F_1(\mu)=1$ and
\begin{align*}
|\widehat f^{kN}_{\omega,\lambda'}(1)-\widehat f^{kN}_{\omega,\mu}(1)|&<\varepsilon\\
|\widehat f^{kN}_{\omega,\lambda'}(1)-\widehat f^{kN}_{\omega,\lambda''}(1)|&<\varepsilon\\
|\widehat f^{kN}_{\omega,\lambda''}(1)-\widehat f^{kN}_{\omega,\mu}(F_2(\mu))|&<\varepsilon,\\
\end{align*}
which implies that
$$
|\widehat f^{kN}_{\omega,\mu}(1)-\widehat f^{kN}_{\omega,\mu}(F_2(\mu))|<3\varepsilon,\qquad\forall k\in\mathbb N,
$$
and thus that $F_2(\mu)=1$ on the open set $U$. By the identity principle it follows that $F_{1,\omega,\lambda''}(\mu)=1$ on $B(\lambda'',\delta)$. By iterating this procedure we conclude that for every $\lambda\in  \Arc[\lambda_0,\lambda']$ we have $1\in J_{\omega,\lambda}$ and $F_{1,\omega,\lambda}(\mu)=1$ on the all ball $B(\lambda,\delta)$. It follows that $1\in J_{\omega,\lambda_0}$.

Recall that by Theorem \ref{l0l1theorem} we have $I_{\lambda_0}= \Arc(1,\lambda_0)$, therefore by Lemma \ref{doublemap} we have $1\in J_{\omega,\lambda_0}$ if and only if $\omega=(d,d,d,\dots)$. Hence for every $\lambda\in  \Arc[\lambda_0,\lambda']$ we must have
$$
f^n_{d,\lambda}(1)\in \mathbb S^1\setminus I_\lambda,
$$
which implies that $1\in J_{d,\lambda}$ for all $\lambda\in  \Arc[\lambda_0,\lambda']$, which contradicts Lemma \ref{nonconstant}.
This concludes the proof of the proposition.
\end{proof}

Theorem \ref{zeroparamsemi} now follows from Montel's Theorem as in Lemma \ref{closureZ}.

\bibliography{mainbib}{}

\providecommand{\bysame}{\leavevmode\hbox to3em{\hrulefill}\thinspace}
\providecommand{\MR}{\relax\ifhmode\unskip\space\fi MR }
\providecommand{\MRhref}[2]{%
  \href{http://www.ams.org/mathscinet-getitem?mr=#1}{#2}
}
\providecommand{\href}[2]{#2}
\begin{thebibliography}{CHJR19}

\bibitem[Bar16]{Ba16}
A.~Barvinok, \emph{Combinatorics and complexity of partition functions},
  Algorithms and Combinatorics, vol.~30, Springer, Cham., 2016.

\bibitem[BG01]{Barata2001}
J.~C.~A. Barata and P.~S. Goldbaum, \emph{On the distribution and gap structure
  of {L}ee--{Y}ang zeros for the {I}sing model: Periodic and aperiodic
  couplings}, Journal of Statistical Physics \textbf{103} (2001), no.~5/6,
  857--891.

\bibitem[BM97]{Barata1997}
J.~C.~A. Barata and D.~H.~U. Marchetti, \emph{Griffiths{\textquotesingle}
  singularities in diluted {I}sing models on the {C}ayley tree}, Journal of
  Statistical Physics \textbf{88} (1997), no.~1-2, 231--268.

\bibitem[CHJR19]{Chio2019}
I.~Chio, C.~He, A.~L. Ji, and R.~K.~W. Roeder, \emph{Limiting measure of
  {L}ee--{Y}ang zeros for the {C}ayley tree}, Communications in Mathematical
  Physics (2019).

\bibitem[DF08]{DF}
R.~Dujardin and C.~Favre, \emph{Distribution of rational maps with a
  preperiodic critical point}, Amer. J. Math. \textbf{130} (2008), no.~4,
  979--1032.

\bibitem[Lev81]{Lev}
G.~M. Levin, \emph{Irregular values of the parameter of a family of polynomial
  mappings}, Uspekhi Mat. Nauk \textbf{36} (1981), no.~6(222), 219--220.

\bibitem[LY52a]{LeeYangPhase1}
T.~D. Lee and C.~N. Yang, \emph{Statistical theory of equations of state and
  phase transitions. {I}. {T}heory of condensation}, Physical Rev. (2)
  \textbf{87} (1952), 404--409. \MR{0053028}

\bibitem[LY52b]{LeeYangPhase2}
\bysame, \emph{Statistical theory of equations of state and phase transitions.
  {II}. {L}attice gas and {I}sing model}, Physical Rev. (2) \textbf{87} (1952),
  410--419. \MR{0053029}

\bibitem[Lyu83]{Ly}
M.~Yu. Lyubich, \emph{Some typical properties of the dynamics of rational
  mappings}, Uspekhi Mat. Nauk \textbf{38} (1983), no.~5(233), 197--198.

\bibitem[McM94]{Mc}
C.~T. McMullen, \emph{Complex dynamics and renormalization}, Annals of
  Mathematics Studies, vol. 135, Princeton University Press, Princeton, NJ,
  1994.

\bibitem[McM00]{Mc2}
\bysame, \emph{The {M}andelbrot set is universal}, The {M}andelbrot set, theme
  and variations, London Math. Soc. Lecture Note Ser., vol. 274, Cambridge
  Univ. Press, Cambridge, 2000, pp.~1--17.

\bibitem[MH77]{Muller-Hartmann1977}
E.~M{\"u}ller-Hartmann, \emph{Theory of the {I}sing model on a {C}ayley tree},
  Zeitschrift f{\"u}r Physik B Condensed Matter \textbf{27} (1977), no.~2,
  161--168.

\bibitem[MHZ75]{Muller-HartmannZittartz1975}
E.~M{\"u}ller-Hartmann and J.~Zittartz, \emph{Phase transitions of continuous
  order: {I}sing model on a {C}ayley tree}, Zeitschrift f{\"u}r {P}hysik {B}
  {C}ondensed {M}atter \textbf{22} (1975), no.~1, 59--67.

\bibitem[Mil00]{Mi}
J.~Milnor, \emph{On rational maps with two critical points}, Experiment. Math.
  \textbf{9} (2000), no.~4, 481--522.

\bibitem[MnSS83]{MSS}
R.~Ma\~{n}\'{e}, P.~Sad, and D.~Sullivan, \emph{On the dynamics of rational
  maps}, Ann. Sci. \'{E}cole Norm. Sup. (4) \textbf{16} (1983), no.~2,
  193--217.

\bibitem[PR17]{PaR17}
V.~Patel and G.~Regts, \emph{Deterministic polynomial-time approximation
  algorithms for partition functions and graph polynomials}, SIAM J. Comput.
  \textbf{46} (2017), no.~6, 1893--1919.

\bibitem[PR18]{PetersRegts2018}
H.~{Peters} and G.~{Regts}, \emph{Location of zeros for the partition function
  of the {I}sing model on bounded degree graphs}, arXiv e-prints (2018),
  arXiv:1810.01699.

\bibitem[PR19]{PR17}
H.~Peters and G.~Regts, \emph{On a conjecture of sokal concerning roots of the
  independence polynomial}, Michigan Math. J. (2019).

\bibitem[Slo91]{Sl}
Z.~Slodkowski, \emph{Holomorphic motions and polynomial hulls}, Proc. Amer.
  Math. Soc. \textbf{111} (1991), no.~2, 347--355.

\bibitem[Sum97]{S1}
H.~Sumi, \emph{On dynamics of hyperbolic rational semigroups}, J. Math. Kyoto
  Univ. \textbf{37} (1997), no.~4, 717--733. \MR{1625944}

\bibitem[Sum98]{S2}
\bysame, \emph{On {H}ausdorff dimension of {J}ulia sets of hyperbolic rational
  semigroups}, Kodai Math. J. \textbf{21} (1998), no.~1, 10--28. \MR{1625124}

\end{thebibliography}
\bibliographystyle{amsalpha}

\end{document}